\documentclass{article}

\usepackage[margin=1.375in, left=1in, right=2.375in]{geometry}

\usepackage{amsmath}
\usepackage{graphicx}
\usepackage{float}
\usepackage{enumerate}
\usepackage{amssymb}
\usepackage{amsfonts}
\usepackage{mathtools}
\usepackage{xcolor}
\usepackage{amsmath}
\usepackage{algorithm,algpseudocode}
\usepackage{subcaption}
\usepackage{caption}
\usepackage{verbatim}
\usepackage{bm}
\usepackage{amsthm} 
\usepackage[toc,page]{appendix}
\usepackage{hyperref} 
\hypersetup{
colorlinks=true,
}
\usepackage[nameinlink,capitalize]{cleveref}
\crefformat{equation}{(#2#1#3)}
\crefrangeformat{equation}{(#3#1#4) to~(#5#2#6)}
\crefmultiformat{equation}{(#2#1#3)}%
{ and~(#2#1#3)}{, (#2#1#3)}{ and~(#2#1#3)}

\usepackage[export]{adjustbox}
\usepackage{rotating}
\usepackage{multirow}
\usepackage{diagbox}
\usepackage{stmaryrd}
\usepackage[utf8]{inputenc}
\usepackage{xstring}
\usepackage{ifthen}
\usepackage{comment}
\usepackage[normalem]{ulem}

\usepackage{cancel}

\newcommand\numberthis{\addtocounter{equation}{1}\tag{\theequation}}

\usepackage{pgf,tikz,pgfplots}
\usetikzlibrary{decorations.markings}
\pgfplotsset{compat=1.15}
\usepackage{mathrsfs}
\usetikzlibrary{arrows}

\newcommand{\rev}[1]{\textcolor{black}{#1}}

\newboolean{hidebool}
\setboolean{hidebool}{false}

\ifthenelse{\boolean{hidebool}}
  {
   \newcommand{\hide}[1]{}%
  }
  {
   \newcommand{\hide}[1]{#1}%
  }

\DeclareMathOperator{\tr}{tr}

\DeclareMathOperator{\rank}{rank}
\DeclareMathOperator{\diag}{diag}
\DeclareMathOperator{\range}{range}

\renewcommand*{\eqref}[1]{(\ref{#1})}
\newcommand{\T}{\textsf{\textup{T}}}
\newcommand{\F}{\textsf{\textup{F}}}
\newcommand{\ALG}{\textsf{\textup{ALG}}}

\newtheorem{theorem}{Theorem}[section]
\newtheorem{corollary}[theorem]{Corollary}
\newtheorem{lemma}[theorem]{Lemma}

\newtheorem*{remark}{Remark}
\theoremstyle{definition}

\title{Randomized block-Krylov subspace methods for low-rank approximation of matrix functions}

\author{David Persson\thanks{New York University \& Flatiron Institute \href{mailto:dup210@nyu.edu}{dup210@nyu.edu}, \href{mailto:dpersson@flatironinstitute.org}{dpersson@flatironinstitute.org}}
\and Tyler Chen\thanks{New York University, \href{mailto:tyler.chen@nyu.edu}{tyler.chen@nyu.edu}} \and Christopher Musco\thanks{New York University, \href{mailto:cmusco@nyu.edu}{cmusco@nyu.edu} }}

\begin{document}

\maketitle

\begin{abstract}
The randomized SVD is a method to compute an inexpensive, yet accurate, low-rank approximation of a matrix. 
The algorithm assumes access to the matrix through matrix-vector products (matvecs).
Therefore, when we would like to apply the randomized SVD to a matrix function, $f(\bm{A})$, one needs to approximate matvecs with $f(\bm{A})$ using some other algorithm, which is typically treated as a black-box. Chen and Hallman (SIMAX 2023) argued that, in the common setting where matvecs with $f(\bm{A})$ are approximated using \emph{Krylov subspace methods} (KSMs), \rev{a} more efficient low-rank approximation is possible if we \emph{open this black-box}. They present an alternative approach that significantly outperforms the naive combination of KSMs with the randomized SVD, although the method lacked theoretical justification.
In this work, we take a closer look at the method, and provide strong and intuitive error bounds that justify its excellent performance for low-rank approximation of matrix functions.\\\\
\textbf{Key words.} Krylov subspace methods, Low-rank approximation, Randomized numerical linear algebra, Matrix functions \\\\
\textbf{MSC Classes.} 65F55, 65F60, 68W20

\end{abstract}
\section{Introduction}\label{section:introduction}

Computing and approximating matrix functions is an important task in many application areas, such as differential equations~\cite{havel1994matrix,highamscaleandsquare}, network analysis~\cite{estrada2000characterization,estradahigham}, and machine learning \cite{DongErikssonNickisch:2017,PleissJankowiakEriksson:2020}. For a symmetric matrix $\bm{A} \in \mathbb{R}^{n \times n}$ with eigenvalue decomposition 
\begin{equation*}
    \bm{A} = \bm{U} \bm{\Lambda} \bm{U}^\T, \quad \bm{\Lambda} = \diag(\lambda_1,\ldots,\lambda_n),
\end{equation*}
the matrix function is defined as 
\begin{equation*}
    f(\bm{A}) = \bm{U} f(\bm{\Lambda}) \bm{U}^\T, \quad f(\bm{\Lambda}) = \diag(f(\lambda_1),\ldots,f(\lambda_n)),
\end{equation*}
where $f$ is a function defined on the eigenvalues of $\bm{A}$. Explicitly forming $f(\bm{A})$ generally costs $O(n^3)$ operations, which becomes prohibitively expensive for large $n$. 

However, instead of computing $f(\bm{A})$ exactly, in many cases we only need an approximation to $f(\bm{A})$, such as a low-rank approximation. For example, recent work has shown that low-rank approximations can be used to accurately estimate the trace or diagonal entries of $f(\bm{A})$ \cite{baston2022stochastic,epperly2023xtrace,jiang2021optimal,hpp, ahpp}. Furthermore, when $f(\bm{A})$ admits an accurate low-rank approximation $\bm{B}$, one can accurately and cheaply approximate matrix-vector products (matvecs) with $f(\bm{A})$ using the approximation
\begin{equation*}
    \bm{B} \bm{Z} \approx f(\bm{A}) \bm{Z}.
\end{equation*}

Popular algorithms to compute cheap, yet accurate, low-rank approximations of matrices include the randomized SVD, randomized subspace iteration, and randomized block Krylov iteration. For an overview and analysis of these algorithms, we turn readers to \cite{rsvd,MM15,tropp2023randomized}. These methods can be referred to as \textit{matrix-free}, since they do not require explicit access to the matrix; they only require implicit access via matvecs. Furthermore, the dominant computational cost of these algorithms is usually the number of matvecs with the matrix that we want to approximate. 

When the matrix we wish to approximate is a matrix-function, $f(\bm{A})$, we do not have access to exact matvecs \rev{even in exact arithmetic}. However, we can typically compute approximate matvecs through some black-box method, such as the (block) Lanczos method, or another Krylov subspace method \cite{blocklanczos,guttel,functionsofmatrices,lanczosfunctionshandbook}. 
These methods are \emph{themselves} matrix-free, only requiring matvecs with $\bm{A}$ to \rev{approximate} matvecs with $f(\bm{A})$. 

To summarize, the standard approach for efficiently computing a low-rank approximation to a matrix function, $f(\bm{A})$, combines two black-box methods: an outer-loop like the randomized SVD that requires matvecs with $f(\bm{A})$, and an inner-loop that computes each of these matvecs using multiple matvecs with $\bm{A}$\rev{.} 
A broad theme of this paper and related work is that computational benefits can often be obtained by \emph{opening up} these black-boxes. I.e., we can get away with \emph{fewer} total matvecs with $\bm{A}$ to obtain a low-rank approximation to $f(\bm{A})$ if we shed the simple inner/outerloop approach.
Such an approach was taken in \cite{chen_hallman_23} and \cite{persson_kressner_23,funnystrom2}, where it is shown that algorithms like the randomized SVD and Nystr\"om approximation for low-rank approximation, and  Hutch++ for trace estimation, can benefit from such introspection.

Concretely, our work builds on a recent paper of
Chen and Hallman \cite{chen_hallman_23}, which develops so-called ``Krylov-aware'' variant of the Hutch++ and A-Hutch++ algorithms for trace estimation \cite{hpp,ahpp} .
The Hutch++ algorithm uses the randomized SVD to compute a low-rank approximation, which in turn is used reduce the variance of a stochastic estimator for the trace of $f(\bm{A})$ \cite{cortinoviskressner,ubarusaad,ubaru2017applications}. \rev{In these settings, even coarse approximations of $f(\bm{A})$ can yield computational benefits. This makes the task of computing such approximations relevant, even when $f(\bm{A})$ does not admit a highly accurate low-rank approximation.}
Chen and Hallman develop a more efficient method to compute this low-rank approximation for matrix functions \rev{and use the method in their version of the Hutch++ algorithm}. 
Numerical experiments suggest that their method is significantly more efficient than naively implementing the randomized SVD on $f(\bm{A})$ with approximate matvecs obtained in a black-box way from a Krylov subspace method. 
\rev{Although their paper is framed around the task of estimating $\tr(f(\bm{A}))$, the key reason for the effectiveness of their method is that the low-rank approximation is more accurate than the randomized SVD. Consequently, to theoretically understand the approximation quality of their method, one must first understand the approximation quality of the low-rank approximation.}

\rev{Our contributions are twofold. First, we strengthen the theoretical foundations of the low-rank approximation method introduced in \cite{chen_hallman_23} by providing explicit and intuitive error bounds that justify its excellent empirical performance. To the best of our knowledge, this is the first work to provide low-rank approximation error bounds for \emph{general} matrix functions; existing theory covers only operator monotone functions \cite{persson_kressner_23,funnystrom2}. Classical theory for Krylov subspace methods have focused on (i) eigenpair approximation, (ii) approximations of the action $f(\bm{A}) \bm{B}$, (iii) low-rank approximation of $\bm{A}$ itself. However, we are not aware of any results that translate existing results to low-rank approximation of matrix-functions. Moreover, our error bounds pave the way to obtain approximation guarantees for estimating $\tr(f(\bm{A}))$.}

\rev{Second, our analysis builds on classical work on Krylov subspace methods for low-rank approximation \cite{drineasipsenkontopoulouismail,rsvd,MM15,tropp2016randomized} and eigenvalue problems \cite{shmuel,paige1971computation,saad1980rates,underwood1975iterative}. We aim to show how this rich literature can be used to obtain approximation guarantees for matrix functions. While we do not claim that the method analyzed here is optimal -- in fact, in \Cref{section:conclusion} we discuss a variant that performs even better in practice, but for which a precise error analysis is currently out of reach -- we view this work as a first step toward a more general theory. Ultimately, we hope that future progress will make it possible to translate classical Krylov theory into general results for low-rank approximation of matrix functions, similar to how \cite{funnystrom2} provided a general ``black-box'' theory that allowed classical results on low-rank approximation to be translated into approximation guarantees for operator monotone matrix functions. }


\begin{remark}\label{section:motivation}
\rev{The bounds we prove in this work contain quantities that depend on the function $f$. In general these terms may be difficult to bound further without any additional assumptions on $f$. A key motivating example of the work by Chen and Hallman \cite{chen_hallman_23} is the matrix exponential $\exp(\bm{A})$ for a symmetric matrix $\bm{A}$. Computing a low-rank approximation to $\exp(\bm{A})$ is a task that arises in numerous applications, ranging from network analysis \cite{hpp}, to quantum thermodynamics \cite{chen_hallman_23,epperly2023xtrace}, to solving PDEs \cite{persson_kressner_23}. 
Note that the exponential has a decaying behavior. Hence, $\exp(\bm{A})$ may admit an accurate low-rank approximation, even when $\bm{A}$ does not. For these reasons, we provide explicit bounds for the exponential and focus our experiments on this case.\footnote{\rev{Other functions, such as $\sqrt{x}, \log(1+x),$ and $\frac{x}{x+1}$, are also relevant in this context. However, these functions are \emph{operator monotone}, and in such cases, alternative methods are more suitable \cite{persson_kressner_23,funnystrom2}.}}}
\end{remark}

\subsection{Notation}
Let $f: \mathbb{R} \mapsto \mathbb{R}$ be a function and let $\bm{A}$ be a symmetric matrix with eigendecomposition
\begin{align}\label{eq:A}
    \begin{split}
    \bm{A} &= \bm{U} \bm{\Lambda} \bm{U}^\T = \begin{bmatrix} \bm{U}_k & \bm{U}_{n \setminus k} \end{bmatrix} \begin{bmatrix} \bm{\Lambda}_k & \\ & \bm{\Lambda}_{n \setminus k} \end{bmatrix} \begin{bmatrix} \bm{U}_k^\T \\ \bm{U}_{n\setminus k}^\T \end{bmatrix}, \\
     \bm{\Lambda}_k &= \diag(\lambda_1,\lambda_2,\cdots,\lambda_k),\\
     \bm{\Lambda}_{n \setminus k} &= \diag(\lambda_{k+1},\lambda_{k+2},\cdots,\lambda_n),
     \end{split}
\end{align}
where the eigenvalues are ordered so that 
\begin{equation*}
|f(\lambda_1)| \geq |f(\lambda_2)| \geq \ldots \geq |f(\lambda_n)|.
\end{equation*}
For a matrix $\bm{B}$ we denote by $\|\bm{B}\|_\F$ and $\|\bm{B}\|_2$ the \emph{Frobenius norm} and \emph{spectral norm}, respectively. 
\rev{For a matrix $\bm{\Omega} \in \mathbb{R}^{n \times \ell}$ we define
\begin{equation}\label{eq:omega_partition}
    \bm{\Omega}_k := \bm{U}_k^\T \bm{\Omega}, \quad \bm{\Omega}_{n\setminus k} := \bm{U}_{n\setminus k}^\T \bm{\Omega}
\end{equation}
Hence, if $\ell \geq k$ and $\rank(\bm{\Omega}_k) = k$ the pseudo-inverse $\bm{\Omega}_k^{\dagger} = \bm{\Omega}_k^\T(\bm{\Omega}_k \bm{\Omega}_k^\T)^{-1}$ satisfies $\bm{\Omega}_k\bm{\Omega}_k^{\dagger} = \bm{I}$.} For any matrix $\bm{B} \in \mathbb{R}^{m \times n}$, the submatrix consisting of rows $r_1$ through $r_2$ and columns $c_1$ through $c_2$ is denoted by $\bm{B}_{r_1:r_2,c_1:c_2}$. We write $\bm{B}_{:,c_1:c_2}$ and $\bm{B}_{r_1:r_2,:}$ to denote $\bm{B}_{1:m,c_1:c_2}$ and $\bm{B}_{r_1:r_2,1:n}$ respectively. We write $\llbracket \bm{B} \rrbracket_k$ to denote the best rank \rev{(at most)} $k$ approximation of $\bm{B}$ obtained by truncating the SVD.
By $\mathbb{P}_s$ we denote the set of all \rev{scalar} polynomials whose degree is at most $s$. 
The Krylov subspace $\mathcal{K}_s(\bm{A},\bm{\Omega})$ is defined as
\begin{equation*}
    \mathcal{K}_{s}(\bm{A},\bm{\Omega}) := \range\left(\begin{bmatrix} \bm{\Omega} & \bm{A} \bm{\Omega} & \cdots & \bm{A}^{s-1} \bm{\Omega} \end{bmatrix}\right).
\end{equation*}
We denote $d_{s,\ell} := \dim(\mathcal{K}_s(\bm{A},\bm{\Omega}))$. For a function $f$ defined on a set $I$ we write
\begin{equation*}
    \|f\|_{L^{\infty}(I)} := \sup\limits_{x \in I}|f(x)|.
\end{equation*}
\section{Krylov-aware low-rank approximation}

We now describe and motivate the algorithm that we will analyse. Inspired by \rev{the work of Chen and Hallman} \cite{chen_hallman_23}, we will call it \textit{Krylov-aware low-rank approximation}. \rev{While the primary goal of Chen and Hallman's work was trace estimation of matrix functions, a central component of their method is the Krylov-aware low-rank approximation method presented in this section.}

We begin with outlining the block-Lanczos algorithm to approximate matvecs with matrix functions. Next, we present how one would naively implement the randomized SVD on $f(\bm{A})$ using the block-Lanczos method to approximate matvecs. Finally, we present the alternative Krylov-aware algorithm and why this method allows us to gain efficiencies. 

\subsection{The block-Lanczos algorithm}\label{section:block_lanczos}

Given an $n\times \ell$ matrix $\bm{\Omega}$, the block-Lanczos algorithm (\cref{alg:block_lanczos}) can be used to iteratively obtain an orthonormal (graded) basis for $\mathcal{K}_{s}(\bm{A},\bm{\Omega})$. 
In particular, using $s$ block-\rev{matvecs} with $\bm{A}$, the algorithm produces \rev{an orthonormal} basis $\bm{Q}_s$ and a block-tridiagonal matrix $\bm{T}_s$ 
\begin{equation}\label{eqn:QTdef}
\bm{Q}_s = 
    \begin{bmatrix} \bm{V}_0 & \cdots & \bm{V}_{s-1} \end{bmatrix}
,\quad
    \bm{T}_s = \bm{Q}_s^\T \bm{A}\bm{Q}_s = \rev{\begin{bmatrix} \bm{M}_1 & \bm{R}_1^\T & & \\
    \bm{R}_1 & \ddots & \ddots & \\
    & \ddots & \ddots & \bm{R}_{s-1}^\T \\
    & & \bm{R}_{s-1} & \bm{M}_s \end{bmatrix}}, 
\end{equation}
where $\bm{R}_0$ is also output by the algorithm and is given by the relation $\bm{\Omega} = \bm{V}_0 \bm{R}_0$. 
\begin{remark}
The block sizes of $\bm{Q}_s$ and $\bm{T}_s$ will never be larger than $\ell$, but may be smaller if the Krylov subspaces $\mathcal{K}_{i}(\bm{A},\bm{\Omega})$ obtained along the way have dimension less than $i\ell$. 
In particular, if $\mathcal{K}_{i}(\bm{A},\bm{\Omega}) = \mathcal{K}_{i-1}(\bm{A},\bm{\Omega})$, then the blocks $\bm{V}_j, \bm{R}_{j-1}, \bm{M}_{j}$ will be empty for all $j \geq i-1$.
The algorithm can terminate at this point, but for analysis it will be useful to imagine that the algorithm is continued and zero-width blocks are produced until it terminates at iteration $s$.
\end{remark}

\begin{algorithm}
\caption{Block-Lanczos Algorithm}
\label{alg:block_lanczos}
\textbf{input:} Symmetric $\bm{A} \in \mathbb{R}^{n \times n}$. Matrix $\bm{\Omega} \in \mathbb{R}^{n \times \ell}$. Number of iterations $s$.\\
\textbf{output:} Orthonormal basis $\bm{Q}_s$ for $\mathcal{K}_{s}(\bm{A}, \bm{\Omega})$,
and block tridiagonal $\bm{T}_s$.
\begin{algorithmic}[1]
    \State Compute an orthonormal basis $\bm{V}_0$ for $\range(\bm{\Omega})$ and $\bm{R}_0 = \bm{V}_0^T \bm{\Omega}$.
    \For{$i=1,\ldots, s$}
    \State $\bm{Y} = \bm{A} \bm{V}_{i-1} - \bm{V}_{i-2} \bm{R}_{i-1}^\T$ \Comment{$\bm{Y} = \bm{A} \bm{V}_{i-1}$ if $i=1$}
    \State $\bm{M}_i = \bm{V}_{i-1}^\T \bm{Y}$ 
    \State $\bm{Y} = \bm{Y} - \bm{V}_{i-1}\bm{M}_i$
    \State $\bm{Y} = \bm{Y} - \sum_{j=0}^{i-1} \bm{V}_j\bm{V}_j^\T \bm{Y}$ \Comment{reorthogonalize (repeat as needed)}
    \State Compute an orthonormal basis $\bm{V}_i$ for $\range(\bm{Y})$ and $\bm{R}_i = \bm{V}_i^T \bm{Y}$. \label{alg:line:qr}
    \EndFor
    \State \Return $\bm{Q}_s$ and $\bm{T}_s$ as in \cref{eqn:QTdef}
\end{algorithmic}
\end{algorithm}


\rev{The Lanczos method has a rich history and has been studied extensively. Single-vector methods (that is, when $\ell = 1$) have classically been studied in the context of approximating a small number of extreme eigenvalues \cite{shmuel,paige1971computation,paige1972computational,scottthesis,krylovschur,wilkinson1971algebraic,thickrestart}, and large-block methods have classically been studied in the context of low-rank approximation \cite{drineasipsenkontopoulouismail, halkomartinssonshkolniskytygert,MM15,tropp2023randomized} and for approximating a few of the largest eigenvalues \cite{cullum,saad1980rates,saad2011numerical,underwood1975iterative}. Moreover, the} block-Lanczos algorithm can \rev{also} be used to approximate matvecs and quadratic forms with $f(\bm{A})$ using the approximations\footnote{\rev{Other approximations are also possible.}}
\begin{align}\label{eq:lanczos_approximation}
    \bm{Q}_s f(\bm{T}_s)_{:,1:\ell}\bm{R}_0 &\approx f(\bm{A}) \bm{\Omega},
    \\
    \label{eq:lanczos_approximation_qf}
    \bm{R}_0^\T f(\bm{T}_s)_{1:\ell,1:\ell}\bm{R}_0 &\approx \bm{\Omega}^\T f(\bm{A}) \bm{\Omega}. 
\end{align}
\rev{Such approximations are often referred to as block-FOM for
functions of matrices \cite{OLeary1980,etna_vol47_pp100-126,frommer_lund_szyld_20} and block-Gauss-quadrature \cite{golub_meurant_09,Zimmerling2025} respectively.}

\rev{It is well-known that if} $f$ is a low-degree polynomial, then the approximations \cref{eq:lanczos_approximation,eq:lanczos_approximation_qf} are exact.\footnote{\eqref{eq:lanczos_approximation} and \eqref{eq:lanczos_approximation_qf} are written out under the assumption that $\bm{\Omega}$ has rank $\ell$. If $\rank(\bm{\Omega}) <\ell$, then the index set $1:\ell$ should be replaced with $1:\rank(\bm{\Omega})$ in both \eqref{eq:lanczos_approximation} and \eqref{eq:lanczos_approximation_qf}.} 
%
\begin{lemma}[\rev{see e.g. \cite[Section 2]{hochbruck_lubich}, \cite[Chapters 2-7]{golub_meurant_09}, \cite[Section 3]{chen_hallman_23}}]\label{lemma:block_lanczos_exact}
The approximation \cref{eq:lanczos_approximation} is exact if $f\in\mathbb{P}_{s-1}$, and the approximation \cref{eq:lanczos_approximation_qf} is exact if $f\in \mathbb{P}_{2s-1}$.
\end{lemma}
It follows that \cref{eq:lanczos_approximation,eq:lanczos_approximation_qf} are good approximations if $f$ is well approximated by polynomials.\footnote{We note that for block-size $\ell>1$, the Krylov subspace is not equivalent to $\cup\{\range(p(\bm{A}) \bm{\Omega}) : p \in \mathbb{P}_{s-1}\}$, and bounds based on best approximation may be pessimistic due to this fact.
In fact, deriving stronger bounds is an active area of research; see e.g. \cite{chen_greenbaum_musco_musco,frommer_schweitzer_guttel,frommer_lund_szyld_20,frommer_schweitzer,frommer_simoncini, hochbruck_lubich}. 
}
In particular, one can obtain bounds in terms of the best polynomial approximation to $f$ on $[\lambda_{\min},\lambda_{\max}]$ \cite{Saad:1992,OrecchiaSachdevaVishnoi:2012,Amsel:2024}.

\subsection{The randomized SVD for matrix functions}
The randomized SVD \cite{rsvd} is a simple and efficient algorithm to compute low-rank approximations of matrices that admit accurate low-rank approximations. The basic idea behind the randomized SVD is that if $\bm{\Omega}$ is a standard Gaussian random matrix, i.e. the entries in $\bm{\Omega}$ are independent identically distributed $N(0,1)$ random variables, then $\range(\bm{B} \bm{\Omega})$ should be \rev{a} reasonable approximation to the range of $\bm{B}$'s top singular vectors. Hence, projecting $\bm{B}$ onto $\range(\bm{B}\bm{\Omega})$ should yield a good approximation to $\bm{B}$. \Cref{alg:rsvd} implements the randomized SVD on a symmetric matrix $\bm{B}$. 
\rev{If $\bm{\Omega}$ has $\ell \geq k$ columns, the} algorithm returns either the rank $\ell \geq k$ approximation $\bm{W}\bm{X} \bm{W}^\T$ or the rank $k$ approximation $\bm{W}\llbracket\bm{X}\rrbracket_k \bm{W}^\T$, depending on the needs of the user.

\begin{algorithm}
\caption{Randomized SVD}
\label{alg:rsvd}
\textbf{input:} Symmetric $\bm{B} \in \mathbb{R}^{n \times n}$. Rank $k$. \rev{Block-size $\ell\geq k$}. \\
\textbf{output:} Low-rank approximation to $\bm{B}$: $\bm{W} \bm{X} \bm{W}^{\T}$ or $\bm{W} \llbracket \bm{X} \rrbracket_k \bm{W}^{\T}$
\begin{algorithmic}[1]
    \State Sample a standard Gaussian $n \times \ell $ matrix $\bm{\Omega}$.
    \State Compute $\bm{K} = \bm{B} \bm{\Omega}$.\label{line:K_rsvd}
    \State Compute an orthonormal basis $\bm{W}$ for $\range(\bm{K})$.\label{line:V}
    \State Compute $\bm{X} = \bm{W}^{\T} \bm{B} \bm{W}$. \label{line:X_rsvd}
    \State \Return $\bm{W} \bm{X} \bm{W}^{\T}$ or $\bm{W} \llbracket \bm{X} \rrbracket_k \bm{W}^{\T}$
\end{algorithmic}
\end{algorithm}

Typically, the dominant cost of \Cref{alg:rsvd} is \rev{the} computation of matvecs \rev{with} $\bm{B}$. We require $2\ell$ such matvecs: $\ell$ in \cref{line:K_rsvd} with $\bm{\Omega}$ and $\ell$ in \cref{line:X_rsvd} with $\bm{W}$. When \Cref{alg:rsvd} is applied to a matrix function $\bm{B} = f(\bm{A})$ these matvecs cannot be performed exactly, but need to be approximated using, for example, the block\rev{-}Lanczos method discussed in the previous section.
\Cref{alg:rsvd_matfun} implements the randomized SVD applied to $f(\bm{A})$ with approximate matvecs using the block\rev{-}Lanczos method. 
The cost is now $(s+r)\ell$ matvecs with $\bm{A}$, where $s$ and $r$ should be set sufficiently large so that the approximations \cref{eq:lanczos_approximation,eq:lanczos_approximation_qf} are accurate.

\begin{algorithm}
\caption{Randomized SVD on a matrix function $f(\bm{A})$}
\label{alg:rsvd_matfun}
\textbf{input:} Symmetric matrix $\bm{A} \in \mathbb{R}^{n \times n}$. Rank $k$. \rev{Block-size $\ell \geq k$}. Matrix function $f: \mathbb{R} \to \mathbb{R}$. Accuracy parameters $s$ and $r$. \\
\textbf{output:} Low-rank approximation to $f(\bm{A})$: $\bm{W}\bm{X} \bm{W}^\T$ or $\bm{W}\llbracket \bm{X}\rrbracket_k \bm{W}^\T$
\begin{algorithmic}[1]
    \State Sample a standard Gaussian $n \times \ell $ matrix $\bm{\Omega}$.
    \State Run \Cref{alg:block_lanczos} for $s$ iterations to obtain an orthonormal basis $\bm{Q}_s$ for $\mathcal{K}_s(\bm{A},\bm{\Omega})$, a block tridiagonal matrix $\bm{T}_s$ and an upper triangular matrix $\bm{R}_0$.
    \State Compute the approximation $\bm{K} = \bm{Q}_s f(\bm{T}_s)_{:,1:\ell} \bm{R}_0 \approx f(\bm{A}) \bm{\Omega}$.
    \State Compute an orthonormal basis $\bm{W}$ for $\range(\bm{K})$. 
    \State Run \Cref{alg:block_lanczos} for $r$ iterations with starting block $\bm{W}$ to obtain the block tridiagonal matrix $\widetilde{\bm{T}}_r$. 
    \State Compute the approximation $\bm{X} = f(\widetilde{\bm{T}}_r)_{1:\ell,1:\ell} \approx \bm{W}^\T f(\bm{A}) \bm{W}$.\label{line:X}
    \State \Return $\bm{W} \bm{X} \bm{W}^{\T}$ or $\bm{W} \llbracket \bm{X} \rrbracket_k \bm{W}^{\T}$
\end{algorithmic}
\end{algorithm}

\subsection{Krylov-aware low-rank approximation}
A key observation in \cite{chen_hallman_23} was that $\range(\bm{W}) \subseteq \range(\bm{Q}_s)$, where $\bm{W}$ and $\bm{Q}_s$ are as in \Cref{alg:rsvd_matfun}. Therefore, by \cite[Lemma 3.3]{funnystrom2} one has
\begin{align*}
    \|f(\bm{A}) - \bm{Q}_s \bm{Q}_s^{\T} f(\bm{A})\bm{Q}_s \bm{Q}_s^{\T}\|_\F &\leq \|f(\bm{A}) - \bm{W} \bm{W}^{\T} f(\bm{A})\bm{W} \bm{W}^{\T}\|_\F,\\
    \|f(\bm{A}) - \bm{Q}_s \llbracket \bm{Q}_s^{\T} f(\bm{A})\bm{Q}_s \rrbracket_k \bm{Q}_s^{\T}\|_\F &\leq \|f(\bm{A}) - \bm{W} \llbracket \bm{W}^{\T} f(\bm{A})\bm{W} \rrbracket_k \bm{W}^{\T}\|_\F.
\end{align*}
Hence, assuming that the quadratic form $\bm{Q}_s^\T f(\bm{A}) \bm{Q}_s$ can be computed accurately, the naive implementation of the randomized SVD outlined in \Cref{alg:rsvd_matfun} will yield a worse error than using $\bm{Q}_s\bm{Q}_s^\T f(\bm{A}) \bm{Q}_s \bm{Q}_s^\T$ as an approximation to $f(\bm{A})$.

Since $\bm{Q}_s$ could have as many as $s\ell$ columns, an apparent downside to this approach is that approximating $f(\bm{A})\bm{Q}_s$ might require $rs\ell$ matvecs with $\bm{A}$ if we use $r$ iterations of the \rev{block-Lanczos} method (\Cref{alg:block_lanczos}) to approximately compute each matvec with $f(\bm{A})$.
The key observation which allows Krylov-aware algorithms to be implemented efficiently is the following:
\begin{lemma}[\rev{see e.g.} {\cite[Section 3]{chen_hallman_23}}]\label{lemma:krylovkrylov}
    Suppose that $\bm{Q}_s$ is the output of \Cref{alg:block_lanczos} with starting block $\bm{\Omega}$ and $s$ iterations. Then, running $r+1$ iterations of \Cref{alg:block_lanczos} with starting block $\bm{Q}_s$ yields the same output as running $s+r$ iterations of \Cref{alg:block_lanczos} with starting block $\bm{\Omega}$. 
    \end{lemma}
The insight behind \Cref{lemma:krylovkrylov} is that, since $\bm{Q}_s$ is a basis for the Krylov subspace $\mathcal{K}_s(\bm{A},\bm{\Omega})$, we have $\mathcal{K}_{r+1}(\bm{A},\bm{Q}_s) = \mathcal{K}_{s+r}(\bm{A},\bm{\Omega})$. This fact can be verified by explicitly writing out $\mathcal{K}_{r+1}(\bm{A},\bm{Q}_s)$:
\begin{align*}
    \mathcal{K}_{r+1}(\bm{A},\bm{Q}_s) &= \range\big( \big[ \bm{Q}_s \,\bm{A}\bm{Q}_s \, \cdots \, \bm{A}^r\bm{Q}_s \big]\big)\\
    &=  \range\big( \big[\bm{\Omega} \,\hspace{5pt}\bm{A}\bm{\Omega} \, \hspace{5pt}\cdots\hspace{5pt} \, \bm{A}^r\bm{\Omega}\\
    &\hspace{2.2cm} \bm{A} \bm{\Omega} \,\hspace{5pt} \bm{A}^2\bm{\Omega} \, \cdots \, \bm{A}^{r+1} \bm{\Omega}\\
    & \hspace{3.5cm} \ddots\\
    & \hspace{3.5cm} \bm{A}^r \bm{\Omega} \,\hspace{5pt}\bm{A}^{r+1} \bm{\Omega} \,\hspace{5pt} \cdots \, \bm{A}^{s+r-1} \bm{\Omega} \big] \big) = \mathcal{K}_{s+r}(\bm{A},\bm{\Omega}). 
\end{align*}
In addition to the work of Chen and Hallman, this observation was recently used to analyze randomized Krylov methods for low-rank approximation \cite{MeyerMuscoMusco:2024,moderateblocksize} and as an algorithmic tool in a randomized variant of the block conjugate gradient algorithm \cite{BCGPreconditioning} and implementations of low-rank approximation algorithms \cite{tropp2023randomized}.
    
Notably, \Cref{lemma:krylovkrylov} enables us to approximate $\bm{Q}_s^\T f(\bm{A})\bm{Q}_s$ with just $r\ell$  additional matvecs with $\bm{A}$, even though $\bm{Q}_s$ has $s\ell$ columns! Hence, approximately computing $\bm{Q}_s^\T f(\bm{A}) \bm{Q}_s$ is essentially no more expensive, in terms of the number of matvecs with $\bm{A}$, than approximating $\bm{W}^\T f(\bm{A}) \bm{W}$, as done in \Cref{alg:rsvd_matfun}~\cref{line:X}. As a consequence of \Cref{lemma:krylovkrylov}, we have the following result. 
\begin{lemma}\label{lemma:2_times_polynomial_approx}
Let $\lambda_{\max}$ and $\lambda_{\min}$ denote the largest and smallest eigenvalue of $\bm{A}$. Let $q = s+r$ and let $\bm{T}_q$ and $\bm{Q}_s$ be computed using \Cref{alg:block_lanczos}. Then, 
\begin{align*}
    &\|\bm{Q}_s^\T f(\bm{A})\bm{Q}_s - f(\bm{T}_q)_{1:d_{s,\ell},1:d_{s,\ell}}\|_\F \leq 2\sqrt{\ell s} \inf\limits_{p \in \mathbb{P}_{2r+1}}\|f(x)-p(x)\|_{L^{\infty}([\lambda_{\min},\lambda_{\max}])}.
\end{align*}
\end{lemma}
\begin{proof}
By \Cref{lemma:block_lanczos_exact,lemma:krylovkrylov} we know that for any polynomial $p \in \mathbb{P}_{2r+1}$ we have $\bm{Q}_s^\T p(\bm{A}) \bm{Q}_{s} = p(\bm{T}_q)_{1:d_{s,\ell},1:d_{s,\ell}}$.
Therefore, since $\|\bm{Q}_s\|_\F \leq \sqrt{\ell s}$ and $\|\bm{Q}_s\|_2 \leq 1$ we have
\begin{align*}
    \hspace{5em}&\hspace{-5em}\|\bm{Q}_s^\T f(\bm{A})\bm{Q}_s - f(\bm{T}_q)_{1:d_{s,\ell},1:d_{s,\ell}}\|_\F
    \\&= \|\bm{Q}_s^\T f(\bm{A}) \bm{Q}_s - p(\bm{T}_q)_{1:d_{s,\ell},1:d_{s,\ell}} + p(\bm{T}_q)_{1:d_{s,\ell},1:d_{s,\ell}} - f(\bm{T}_q)_{1:d_{s,\ell},1:d_{s,\ell}} \|_\F 
    \\&\leq \|\bm{Q}_s^\T f(\bm{A}) \bm{Q}_s - \bm{Q}_s^\T p(\bm{A}) \bm{Q}_s \|_\F +  \|(p(\bm{T}_q) - f(\bm{T}_q) )_{1:d_{s,\ell},1:d_{s,\ell}} \|_\F 
    \\&\leq \sqrt{\ell s} \left( \|f(\bm{A}) - p(\bm{A}) \|_2 +  \| p(\bm{T}_q)  - f(\bm{T}_q) \|_2\right)
    \\&\leq 2 \sqrt{\ell s} \|f(x) - p(x)\|_{L^{\infty}([\lambda_{\min},\lambda_{\max}])},
\end{align*}
where the last inequality is due to the fact that the spectrum of $\bm{T}_q$ is contained in $[\lambda_{\min},\lambda_{\max}]$.
Optimizing over $p\in\mathbb{P}_{2r+1}$ gives the result.
\end{proof}

We can now present the Krylov-aware low-rank approximation algorithm\rev{, which was presented in \cite[Section 3]{chen_hallman_23}}; see \cref{alg:krylow}. The total number of matvecs with $\bm{A}$ is $(s+r)\ell$, the same as \cref{alg:rsvd_matfun}.\footnote{\rev{Note that \cref{alg:krylow} requires computing a matrix-function of size $(s+r) \ell \times (s+r)\ell$, which is more expensive than computing the two matrix functions of size $s \ell \times s \ell$ and $r\ell \times r \ell$ in \Cref{alg:rsvd_matfun}. However, if $s,r \ll n$, the cost of these matrix-function evaluations is not the dominant part of the computation.}} 
However, as noted above, \cref{alg:krylow} (approximately) projects $f(\bm{A})$ onto a higher dimensional \rev{subspace}, ideally obtaining a better approximation.


\begin{algorithm}
\caption{Krylov-aware low-rank approximation}
\label{alg:krylow}
\textbf{input:} Symmetric $\bm{A} \in \mathbb{R}^{n \times n}$. Rank $k$. 
\rev{Block-size $\ell \geq k$}. 
Matrix function $f: \mathbb{R} \to \mathbb{R}$. Number of iterations $q = s + r$.\\
\textbf{output:} Low-rank approximation to $f(\bm{A})$: $\bm{Q}_s \bm{X} \bm{Q}_s^\T$ or $\bm{Q}_s \llbracket \bm{X} \rrbracket_k  \bm{Q}_s^\T$
\begin{algorithmic}[1]
    \State Sample a standard Gaussian $n \times \ell $ matrix $\bm{\Omega}$.
    \State Run \cref{alg:block_lanczos} for $q=s+r$ iterations to obtain an orthonormal basis $\bm{Q}_s$ for $\mathcal{K}_{s}(\bm{A},\bm{\Omega})$ and a block tridiagonal matrix $\bm{T}_q$.
    \State Compute $\bm{X} = f(\bm{T}_q)_{1:d_{s,\ell},1:d_{s,\ell}}$ \rev{where $d_{s,\ell} = \dim(\mathcal{K}_s(\bm{A},\bm{\Omega}))$}. \Comment{$\approx \bm{Q}_s^\T f(\bm{A}) \bm{Q}_s$} 
    \State \textbf{return} $\textsf{\textup{ALG}}(s,r;f) = \bm{Q}_s \bm{X}\bm{Q}_s^\T$ or $\textsf{\textup{ALG}}_k(s,r;f) = \bm{Q}_s \llbracket \bm{X} \rrbracket_k  \bm{Q}_s^\T$.
\end{algorithmic}
\end{algorithm}
    

\rev{\Cref{alg:krylow} requires $\ell(s+r)$ matvecs with $\bm{A}$ and $O(n(s+r)^2 \ell^2)$ additional operations, assuming full reorthogonalization is performed in \Cref{alg:block_lanczos}. For dense matrices, the cost is dominated by matvecs with $\bm{A}$, resulting in an overall complexity of $O(n^2(s+r)\ell)$ operations. In contrast, when $\bm{A}$ is sparse, the reorthogonalization becomes the dominant computational cost and scales quadratically with the block-size $\ell$. While \Cref{alg:rsvd} and \Cref{alg:rsvd_matfun} require the block-size $\ell$ to be larger than $k$ to be able to return a rank-$k$ approximation, \Cref{alg:krylow} can return a rank-$k$ approximation even if $\ell < k$, provided $q\ell \geq k$. Recent work has shown that single-vector methods (i.e., setting $\ell = 1$) can often produce a more efficient low-rank approximation at the same computational cost. Furthermore, Chen and Hallman \cite{chen_hallman_23} recommended using a small block-size, say $\ell = 1,2,$ or $4$. Later in this work we will consider a version of \Cref{alg:krylow} which uses a single-vector starting block, i.e. $\ell = 1$. However, its analysis relies on the analysis of the block-version in \Cref{alg:krylow}, which restricts $\ell \geq k$. We therefore defer the discussion on a single-vector version of \Cref{alg:krylow} to \Cref{section:singlevector}. Moreover, while single-vector methods often produce more accurate approximations at the same number of arithmetic operations, large-block methods come with other advantages, such as being more parallelizable. }

We conclude by noting that the function $f$ in \cref{alg:krylow} does not need to be fixed; one can compute a low-rank approximation for many different functions $f$ at minimal additional cost. 

\begin{remark}
    \rev{Modern methods for eigenpair approximation often employ restarting techniques to reduce memory requirements. Furthermore, Chen and Hallman does discuss a restarting strategy \cite[Section 4.2]{chen_hallman_23} based on classical techniques \cite{sorensen,baglama}. While restarting techniques are well studied in the context of eigenpair approximations, they are less understood in the context of low-rank approximation. In this work, the theoretical analysis will not focus on restarting techniques.}
\end{remark}

\section{Error bounds}
In this section, we establish error bounds for \Cref{alg:krylow}. We break the analysis into two parts. First, in \Cref{section:inexact_projections}, we derive general error bounds for approximations to $f(\bm{A})$ when projections $\bm{Q} \bm{Q}^\T f(\bm{A}) \bm{Q} \bm{Q}^\T$ cannot be computed exactly. In \Cref{section:structural} we provide structural bounds for the errors $\|f(\bm{A})- \bm{Q}_s \bm{Q}_s^\T f(\bm{A}) \bm{Q}_s \bm{Q}_s^\T\|_{\F}$ and $\|f(\bm{A})- \bm{Q}_s \llbracket \bm{Q}_s^\T f(\bm{A}) \bm{Q}_s\rrbracket_k \bm{Q}_s^\T\|_{\F}$ that hold with probability $1$, and in \Cref{section:probabilistic} we derive the corresponding probabilistic bounds. Next, in \Cref{section:krylow} we combine the results from \Cref{section:inexact_projections,section:structural,section:probabilistic} to derive end-to-end error bounds for \Cref{alg:krylow}, which involves approximate projection onto $\bm{Q}_s \bm{Q}_s^\T$. 
\subsection{Error bounds for inexact projections}\label{section:inexact_projections}
In this section we will derive error bounds for $\|f(\bm{A}) - \bm{Q} \bm{X} \bm{Q}^\T\|_{\F}$ and $\|f(\bm{A}) - \bm{Q} \llbracket \bm{X} \rrbracket_k \bm{Q}^\T\|_{\F}$ where $\bm{Q}$ is \textit{any} orthonormal basis and $\bm{X}$ is \textit{any} matrix. By \cite[Lemma 3.3]{funnystrom2} we know that the optimal choice of $\bm{X}$ is $\bm{X} = \bm{Q}^\T f(\bm{A}) \bm{Q}$. However, since $\bm{Q}^\T f(\bm{A}) \bm{Q}$ can only be computed approximately, we need to show that the errors $\|f(\bm{A}) - \bm{Q} \bm{Q}^\T f(\bm{A}) \bm{Q} \bm{Q}\|_{\F}$ and $\|f(\bm{A}) - \bm{Q} \llbracket \bm{Q}^\T f(\bm{A}) \bm{Q} \rrbracket_k \bm{Q}\|_{\F}$ are robust against perturbations in $\bm{Q}^\T f(\bm{A}) \bm{Q}$. \rev{Robustness of the former follows directly from the Pythagorean theorem
\begin{equation}\label{eq:robustness1}
        \|f(\bm{A}) - \bm{Q} \bm{X}\bm{Q}^\T\|_{\F}^2 = \|f(\bm{A}) - \bm{Q}\bm{Q}^\T f(\bm{A})\bm{Q}\bm{Q}^\T\|_{\F}^2 + \|\bm{Q}^\T f(\bm{A})\bm{Q} - \bm{X}\|_{\F}^2.
    \end{equation}
    Establishing robustness of the latter requires more care, especially since rank-$k$ truncations are not Lipschitz continuous. However, we show in \Cref{theorem:robust} that it is indeed robust.}
\begin{theorem}\label{theorem:robust}
    Given an orthonormal basis \rev{$\bm{Q} \in \mathbb{R}^{n \times d}$} and an approximation $\rev{\bm{X} \in \mathbb{R}^{d \times d}}$ to $\bm{Q}^\T f(\bm{A}) \bm{Q}$. Then,
    \begin{equation}\label{eq:robustness2}
        \rev{\|f(\bm{A}) - \bm{Q} \llbracket \bm{X} \rrbracket_k \bm{Q}^\T\|_{\F} \leq \|f(\bm{A}) - \bm{Q}\llbracket\bm{Q}^\T f(\bm{A})\bm{Q} \rrbracket_k\bm{Q}^\T\|_{\F} + 2\|\bm{Q}^\T f(\bm{A})\bm{Q} - \bm{X}\|_{\F}.}
    \end{equation}
\end{theorem}
\begin{proof}

\rev{We will prove \eqref{eq:robustness2}} using a similar argument to \cite[Proof of Theorem 5.1]{tropp2016randomized}. Define $\bm{C} = f(\bm{A}) - \bm{Q} \bm{Q}^\T f(\bm{A}) \bm{Q} \bm{Q}^\T + \bm{Q} \bm{X}\bm{Q}^\T$. Note that $\|\bm{C} - f(\bm{A})\|_{\F} = \|\bm{Q}^\T f(\bm{A}) \bm{Q} - \bm{X}\|_{\F}$ and $\bm{Q}^\T \bm{C} \bm{Q} =  \bm{X}$. Hence,
    \begin{align*}
        \|f(\bm{A}) - \bm{Q} \llbracket\bm{X}\rrbracket_k\bm{Q}^\T\|_{\F} &= \|f(\bm{A}) - \bm{Q}\llbracket \bm{Q}^\T \bm{C} \bm{Q}\rrbracket_k\bm{Q}^\T\|_{\F} 
        \\&\leq 
        \|f(\bm{A}) - \bm{C}\|_{\F} + \|\bm{C} - \bm{Q}\llbracket \bm{Q}^\T \bm{C} \bm{Q}\rrbracket_k\bm{Q}^\T\|_{\F} 
        \\&=
        \|\bm{Q}^\T f(\bm{A})\bm{Q}-\bm{X}\|_{\F} + \|\bm{C} - \bm{Q}\llbracket \bm{Q}^\T \bm{C} \bm{Q}\rrbracket_k\bm{Q}^\T\|_{\F}. \numberthis \label{eq:first_ineq}
    \end{align*}
    Since $\bm{Q}\llbracket \bm{Q}^\T \bm{C} \bm{Q}\rrbracket_k\bm{Q}^\T$ is the best rank $k$ approximation whose range is contained in $\range(\bm{Q})$ \cite[Lemma 3.3]{funnystrom2} (a similar result is given in \cite[Theorem 3.5]{gu_subspace}), we have
    \begin{align*}
        \|\bm{C} - \bm{Q}\llbracket \bm{Q}^\T \bm{C} \bm{Q}\rrbracket_k\bm{Q}^\T\|_{\F} 
         \leq&\|\bm{C} - \bm{Q}\llbracket \bm{Q}^\T f(\bm{A}) \bm{Q}\rrbracket_k\bm{Q}^\T\|_{\F} \\
         \leq& \|\bm{Q}^\T f(\bm{A})\bm{Q}-\bm{X}\|_{\F} + \|f(\bm{A}) - \bm{Q}\llbracket \bm{Q}^\T f(\bm{A}) \bm{Q}\rrbracket_k\bm{Q}^\T\|_{\F}\numberthis \label{eq:final_ineq},
    \end{align*}
    where the second inequality is due to the fact that $\|f(\bm{A})-\bm{C}\|_\F =\|\bm{Q}^\T f(\bm{A})\bm{Q}-\bm{X}\|_{\F}$. Combining \eqref{eq:first_ineq} and \eqref{eq:final_ineq} yields the desired inequality.
\end{proof}

A corollary of \Cref{theorem:robust} is that the error of the approximation from \Cref{alg:krylow} will always be bounded from above by the error of the approximation from \Cref{alg:rsvd_matfun}, up to a polynomial approximation error of $f$:
\begin{corollary}
\label{corr:early_bigger_subspace}
    Let $\lambda_{\max}$ and $\lambda_{\min}$ be the largest and smallest eigenvalues of $\bm{A}$. Let $\bm{Q}_s$ and $\bm{X}$ be the output from \Cref{alg:krylow} and let $\bm{W}$ and $\widetilde{\bm{X}}$ be the output from \Cref{alg:rsvd_matfun} with the same input parameters and the same sketch matrix $\bm{\Omega}$. Then,
    \begin{align*}
        &\|f(\bm{A}) - \bm{Q}_s \llbracket\bm{X}\rrbracket_k \bm{Q}_s^\T\|_\F 
        \\&\hspace{2em}
        \leq \|f(\bm{A}) - \bm{W} \llbracket\widetilde{\bm{X}}\rrbracket_k \bm{W}^\T\|_\F + 4\sqrt{\ell s}  \inf\limits_{p \in \mathbb{P}_{2r+1}}\|f(x)-p(x)\|_{L^{\infty}([\lambda_{\min},\lambda_{\max}])}.
    \end{align*}
\end{corollary}
\begin{proof}
    Since $\range(\bm{Q}_s) \subseteq \range(\bm{W})$ we have
    \begin{align*}
        &\|f(\bm{A}) - \bm{Q}_s\llbracket\bm{Q}_s^\T f(\bm{A})\bm{Q}_s \rrbracket_k\bm{Q}_s^\T\|_{\F} 
        \\&\hspace{7em}\leq \|f(\bm{A}) - \bm{W}\llbracket\bm{W}^\T f(\bm{A})\bm{W} \rrbracket_k\bm{W}^\T\|_{\F} 
        \\&\hspace{14em}\leq 
        \|f(\bm{A}) - \bm{W}\llbracket\widetilde{\bm{X}}\rrbracket_k\bm{W}^\T\|_{\F},
    \end{align*}
    where we used that $\bm{W}\llbracket\bm{W}^\T f(\bm{A})\bm{W} \rrbracket_k\bm{W}^\T$ is the best rank $k$ approximation to $f(\bm{A})$ whose range and co-range is contained in $\range(\bm{W})$ \cite[Lemma 3.3]{funnystrom2}. Hence, by \Cref{theorem:robust} we have
    \begin{align*}
        \|f(\bm{A}) - \bm{Q}_s \llbracket \bm{X} \rrbracket_k \bm{Q}_s^\T\|_{\F} 
        &\leq \|f(\bm{A}) - \bm{Q}_s\llbracket\bm{Q}_s^\T f(\bm{A})\bm{Q}_s \rrbracket_k\bm{Q}_s^\T\|_{\F} + 2\|\bm{Q}_s^\T f(\bm{A})\bm{Q}_s - \bm{X}\|_{\F} 
        \\&\leq 
        \|f(\bm{A}) - \bm{W}\llbracket\widetilde{\bm{X}}\rrbracket_k\bm{W}^\T\|_{\F} + 2\|\bm{Q}_s^\T f(\bm{A})\bm{Q}_s - \bm{X}\|_{\F}.
    \end{align*}
    Applying \Cref{lemma:2_times_polynomial_approx} yields the desired result. 
\end{proof}
\Cref{corr:early_bigger_subspace} already shows why we expect the Krylov-aware approach in \Cref{alg:krylow} to outperform a naive combination of Krylov subspace methods and the randomized SVD in \Cref{alg:rsvd_matfun}. In fact, we might expect a major improvement, since $\range(\bm{Q}_s)$ is a significantly larger subspace than $\range(\bm{W})$. In the subsequent sections we will derive stronger bounds for the approximation returned by \Cref{alg:krylow} to better justify this intuition.

\subsection{Structural bounds}\label{section:structural}
We next derive structural bounds for $\|f(\bm{A})- \bm{Q}_s \bm{Q}_s^\T f(\bm{A}) \bm{Q}_s \bm{Q}_s^\T\|_{\F}$ and $\|f(\bm{A})- \bm{Q}_s \llbracket \bm{Q}_s^\T f(\bm{A}) \bm{Q}_s\rrbracket_k \bm{Q}_s^\T\|_{\F}$ that is true for \textit{any} sketch matrix $\bm{\Omega}$ as long as $\bm{\Omega}_k$ defined in \eqref{eq:omega_partition} has rank $k$. These bounds will allow us to obtain probabilistic bounds on the error of approximating $f(\bm{A})$, at least under the assumption that $\bm{Q}_s \bm{Q}_s^\T f(\bm{A}) \bm{Q}_s \bm{Q}_s^\T$ and $\bm{Q}_s \llbracket \bm{Q}_s^\T f(\bm{A}) \bm{Q}_s\rrbracket_k \bm{Q}_s^\T$ are computed exactly. As a reminder, we will remove this assumption in \Cref{section:krylow} using the perturbation bounds from \Cref{section:inexact_projections}.

To state our bounds, we introduce the quantity
\begin{equation}
\label{eqn:min_ratio_omega}
    \mathcal{E}_{\bm{\Omega}}(s;f) =
    \min\limits_{p \in \mathbb{P}_{s-1}}\left[\|p(\bm{\Lambda}_{n \setminus k}) \bm{\Omega}_{n \setminus k} \bm{\Omega}_k^{\dagger}\|_{\F}^2\max\limits_{i=1,\ldots,k} \left|\frac{f(\lambda_i)}{p(\lambda_i)}\right|^2\right],
\end{equation}
which quantifies the extent to which an $(s-1)$ degree polynomial can be large (relative to $f$) on the eigenvalues $\lambda_1, \ldots, \lambda_k$ and small on the remaining eigenvalues; \rev{similar quantities have appeared in other contexts \cite{RUHE1984391,guttelphd}.} Recall that we order $\bm{A}$'s eigenvalues with respect to $f$, so $|f(\lambda_1)| \geq |f(\lambda_1)| \geq \ldots \geq |f(\lambda_n)|$.
\begin{lemma}\label{lemma:structural}
    Consider $\bm{A} \in \mathbb{R}^{n \times n}$ as defined in \eqref{eq:A} \rev{and let $\bm{Q}_s$ be an orthonormal basis for $\mathcal{K}_{s}(\bm{A},\bm{\Omega})$}. Assuming $\bm{\Omega}_k$ in \eqref{eq:omega_partition} has rank $k$, for all functions $f: \mathbb{R} \to \mathbb{R}$, we have
    \begin{align}\label{eq:structural}
        \begin{split}
        &\|f(\bm{A}) - \bm{Q}_s  \bm{Q}_s^\T f(\bm{A})\bm{Q}_s \bm{Q}_s^\T \|_{\F}^2 
        \\&\hspace{7em}\leq \|f(\bm{A}) - \bm{Q}_s \llbracket \bm{Q}_s^\T f(\bm{A})\bm{Q}_s\rrbracket_k \bm{Q}_s^\T\|_{\F}^2 
        \\&\hspace{14em}\leq
        \|f(\bm{\Lambda}_{n \setminus k})\|_{\F}^2 + 5 \mathcal{E}_{\bm{\Omega}}(s;f).
        \end{split}
    \end{align}
\end{lemma}

From this bound, we can further see why \Cref{alg:krylow} should be preferred over directly applying the randomized SVD to $f(\bm{A})$. As an extreme case, consider when $f$ is a polynomial of degree at most $s-1$. Then the standard randomized SVD bound \cite[Theorem 9.1]{rsvd} essentially replaces $p$ with $f$ in $\mathcal{E}_{\bm{\Omega}}(s;f)$ in \eqref{eq:structural}:
\begin{equation*}
    \|f(\bm{A}) - \bm{W} \llbracket \bm{W}^{\rev{\T}} f(\bm{A}) \bm{W} \rrbracket_k \bm{W}\|_\F^2 \leq \|f(\bm{\Lambda}_{n \setminus k})\|_\F^2 + 5 \|f(\bm{\Lambda}_{n \setminus k}) \bm{\Omega}_{n \setminus k} \bm{\Omega}_k^{\dagger}\|_\F^2,
\end{equation*}
where $\bm{W}$ is an orthonormal basis for $\range(f(\bm{A}) \bm{\Omega})$. Since $\mathcal{E}_{\bm{\Omega}}(s;f)$ minimizes over \emph{all} degree $s-1$ polynomials, it is always smaller than  $\|f(\bm{\Lambda}_{n \setminus k}) \bm{\Omega}_{n \setminus k} \bm{\Omega}_k^{\dagger}\|_\F^2$. So, the approximation returned by \Cref{alg:krylow} is expected to be more accurate compared to the randomized SVD. When $f$ is not a polynomial, the error of the randomized SVD roughly corresponds to plugging a good polynomial approximation for $f$ into \cref{eqn:min_ratio_omega}, but again there might  be a better choice to minimize $\mathcal{E}_{\bm{\Omega}}(s;f)$. 

Indeed, in many cases, we find that the improvement obtained from effectively optimizing over all possible degree $s-1$ polynomials is significant -- there is often a much better choice of polynomial than a direct approximation to $f$ for \cref{eqn:min_ratio_omega}. This is reflected in our experiments (\Cref{sec:experiments}) and we also provide a concrete analysis involving the matrix exponential in \Cref{section:exponential}. 

\rev{We also highlight that this bound shows that if there is a polynomial $p \in \mathbb{P}_{s-1}$ that is large on $\lambda_1, \ldots, \lambda_k$ and small on $\lambda_{k+1},\ldots,\lambda_{n}$, then one can construct a near-optimal rank $k$ approximation whose range is contained in $\range(\bm{Q}_s)$; similar ideas have been used in the context of low-rank approximations of $\bm{A}$ \cite{MM15,tropp2023randomized} and eigenvalue approximation \cite{paige1971computation,saad1980rates,underwood1975iterative}. In particular, we expect that $\mathcal{E}(s;f)$ will decay rapidly as $s$ increases when $f$ is monotonically increasing or decreasing, since in these cases one can find explicit polynomials to upper bound $\mathcal{E}_{\bm{\Omega}}(s;f)$. For example, when $f$ is monotonic one can choose the polynomial to be a scaled and shifted Chebyshev polynomial. }

\begin{proof}[Proof of \Cref{lemma:structural}]
The first inequality is due to the fact that $\bm{Q}_s\bm{Q}_s^{\T}f(\bm{A})\bm{Q}_s\bm{Q}_s^{\T}$ is the nearest matrix to $f(\bm{A})$ in the Frobenius norm whose range and co-range is contained in $\range(\bm{Q}_s)$ \cite[Lemma 3.3]{funnystrom2}. 
We proceed with proving the second inequality. 

To do so, it of course suffices to prove the inequality where $\mathcal{E}_{\bm{\Omega}}(s;f)$ is replaced with $\|p(\bm{\Lambda}_{n \setminus k}) \bm{\Omega}_{n \setminus k} \bm{\Omega}_k^{\dagger}\|_{\F}^2\cdot \max_{i=1,\ldots,k} \left|{f(\lambda_i)}/{p(\lambda_i)}\right|^2$ for any choice of $p \in \mathbb{P}_{s-1}$.
Note that if we choose $p$ so that $p(\lambda_i) = 0$ for some $i = 1,\ldots,k$ then the right hand side of \eqref{eq:structural} is infinite and the bound trivially holds. 
Hence, we may assume that $p(\lambda_i) \neq 0$ for $i = 1,\ldots,k$.
Consequently, $p(\bm{\Lambda}_k)$ is non-singular.
Define $\bm{Z} = p(\bm{A}) \bm{\Omega} \bm{\Omega}_k^{\dagger} p(\bm{\Lambda}_k)^{-1}$ and let $\widetilde{\bm{P}}$ be the orthogonal projector onto $\range(\bm{Z}) \subseteq \range(\bm{Q}_s)$.
Note that $\rank(\bm{Z}) \leq k$ and $\bm{Q}_s \llbracket\bm{Q}_s^\T f(\bm{A})\bm{Q}_s\rrbracket_{k} \bm{Q}_s^\T$ is the best rank $k$ approximation to $f(\bm{A})$ whose range and co-range is contained in $\range(\bm{Q}_s)$ \cite[Lemma 3.3]{funnystrom2}.
Hence,
\begin{equation*}
    \|f(\bm{A}) - \bm{Q}_s \llbracket\bm{Q}_s^\T f(\bm{A})\bm{Q}_s\rrbracket_{k} \bm{Q}_s^\T \|_{\F}^2 \leq \|f(\bm{A}) - \widetilde{\bm{P}}f(\bm{A})\widetilde{\bm{P}} \|_{\F}^2.
\end{equation*}
Now define $\widehat{\bm{P}} = \bm{U}^\T \widetilde{\bm{P}} \bm{U}$, which is the orthogonal projector onto $\range(\bm{U}^\T \bm{Z})$. By the unitary invariance of the Frobenius norm we have
\begin{equation*}
    \|f(\bm{A}) - \widetilde{\bm{P}}f(\bm{A})\widetilde{\bm{P}} \|_{\F}^2 = \|f(\bm{\Lambda}) - \widehat{\bm{P}} f(\bm{\Lambda}) \widehat{\bm{P}}\|_{\F}^2.
\end{equation*}
Furthermore, by Pythagorean theorem,
\begin{equation}\label{eq:two_terms}
    \|f(\bm{\Lambda}) - \widehat{\bm{P}} f(\bm{\Lambda}) \widehat{\bm{P}}\|_{\F}^2 = \|(\bm{I} - \widehat{\bm{P}}) f(\bm{\Lambda})\|_{\F}^2 + \|\widehat{\bm{P}} f(\bm{\Lambda})(\bm{I}-\widehat{\bm{P}})\|_{\F}^2.
\end{equation}
We are going to bound the two terms on the right hand side of \eqref{eq:two_terms} separately. 
Our analysis for the first is similar to the proof of \cite[Theorem 9.1]{rsvd}.

Note that since $\rank(\bm{\Omega}_k) = k$ we have $\bm{\Omega}_k \bm{\Omega}_k^{\dagger} = \bm{I}$. Hence,
\begin{equation*}
    \bm{U}^\T\bm{Z} = \bm{U}^\T p(\bm{A}) \bm{\Omega} \bm{\Omega}_k^{\dagger} p(\bm{\Lambda}_k)^{-1} = \begin{bmatrix} \bm{I} \\ p(\bm{\Lambda}_{n \setminus k}) \bm{\Omega}_{n \setminus k} \bm{\Omega}_k^{\dagger} p(\bm{\Lambda}_k)^{-1}\end{bmatrix} =: \begin{bmatrix} \bm{I} \\ \bm{F} \end{bmatrix}. 
\end{equation*}
Hence, 
\begin{align*}
    \bm{I}-\widehat{\bm{P}} &= \begin{bmatrix} \bm{I} - (\bm{I}+\bm{F}^\T \bm{F})^{-1} & -(\bm{I}+\bm{F}^\T \bm{F})^{-1}\bm{F}^\T \\ -\bm{F}(\bm{I}+\bm{F}^\T \bm{F})^{-1} & \bm{I} - \bm{F}(\bm{I}+\bm{F}^\T \bm{F})^{-1}\bm{F}^\T\end{bmatrix} \\
    &\preceq \begin{bmatrix} \bm{F}^\T \bm{F} & -(\bm{I}+\bm{F}^\T \bm{F})^{-1}\bm{F}^\T \\ -\bm{F}(\bm{I}+\bm{F}^\T \bm{F})^{-1} & \bm{I}\end{bmatrix},
\end{align*}
\rev{where $\preceq$ denotes the Loewner order \cite[Definition 7.7.1]{matrixanalysis} and} the inequality is due to \cite[Proposition 8.2]{rsvd}. \rev{Hence, we have $f(\bm{\Lambda}_1)\left(\bm{I}-(\bm{I} + \bm{F}^\T \bm{F})^{-1}\right) f(\bm{\Lambda}_1) \preceq f(\bm{\Lambda}_1)\bm{F}^\T \bm{F} f(\bm{\Lambda}_1)$ and $f(\bm{\Lambda}_2)\left(\bm{I}-\bm{F}(\bm{I} + \bm{F}^\T \bm{F})^{-1}\bm{F}^\T\right) f(\bm{\Lambda}_2) \preceq f(\bm{\Lambda}_2)^2$.} Consequently, \rev{using \cite[Corollary 7.7.4]{matrixanalysis}, the definition of $\bm{F}$, and strong submultiplicativity of the Frobenius norm we get}
\begin{align*}
     \|(\bm{I} - \widehat{\bm{P}}) f(\bm{\Lambda})\|_{\F}^2 &= \tr(f(\bm{\Lambda})(\bm{I} - \widehat{\bm{P}})f(\bm{\Lambda})) 
     \\&\leq
      \|f(\bm{\Lambda}_{n \setminus k})\|_{\F}^2 + \|\bm{F}f(\bm{\Lambda}_k)\|_{\F}^2 \\&\leq
     \|f(\bm{\Lambda}_{n \setminus k})\|_{\F}^2 + \|p(\bm{\Lambda}_{n \setminus k})\bm{\Omega}_{n \setminus k} \bm{\Omega}_k^{\dagger}\|_{\F}^2\|p(\bm{\Lambda}_k)^{-1}f(\bm{\Lambda}_k)\|_2^2  
     \\&= \|f(\bm{\Lambda}_{n \setminus k})\|_{\F}^2 + \|p(\bm{\Lambda}_{n \setminus k})\bm{\Omega}_{n \setminus k} \bm{\Omega}_k^{\dagger}\|_{\F}^2 \max\limits_{i=1,\ldots,k} \left|\frac{f(\lambda_i)}{p(\lambda_i)}\right|^2. \numberthis \label{eq:first_term}
\end{align*}

We proceed with bounding the second term in \eqref{eq:two_terms}.
Our analysis is similar to the proof of \cite[Lemma 3.7]{persson_kressner_23}. 
By the triangle inequality we have 
\begin{equation*}
    \|\widehat{\bm{P}} f(\bm{\Lambda})(\bm{I}-\widehat{\bm{P}})\|_{\F} \leq \left\|\begin{bmatrix} \bm{0} & \\ & f(\bm{\Lambda}_{n \setminus k})\end{bmatrix} \widehat{\bm{P}}\right\|_{\F} + \left\|(\bm{I} - \widehat{\bm{P}}) \begin{bmatrix} f(\bm{\Lambda}_k) & \\ & \bm{0}\end{bmatrix}\right\|_{\F}.
\end{equation*}
Using a similar argument as in \eqref{eq:first_term} we have
\begin{equation}
    \left\|(\bm{I} - \widehat{\bm{P}}) \begin{bmatrix} f(\bm{\Lambda}_k) & \\ & \bm{0}\end{bmatrix} \right\|_{\F} \leq \|p(\bm{\Lambda}_{n \setminus k})\bm{\Omega}_{n \setminus k} \bm{\Omega}_k^{\dagger}\|_{\F}\max\limits_{i=1,\ldots,k} \left|\frac{f(\lambda_i)}{p(\lambda_i)}\right|,\label{eq:second_term}
\end{equation}
and since $\bm{F}(\bm{I}+\bm{F}^\T \bm{F})^{-1}\bm{F}^\T \preceq \bm{F}\bm{F}^\T$ we have
\begin{align*}
    \left\|\begin{bmatrix} \bm{0} & \\ & f(\bm{\Lambda}_{n \setminus k})\end{bmatrix} \widehat{\bm{P}}\right\|_{\F}^2 
    &= \tr\left(\begin{bmatrix} \bm{0} & \\ & f(\bm{\Lambda}_{n \setminus k})\end{bmatrix} \widehat{\bm{P}}\begin{bmatrix} \bm{0} & \\ & f(\bm{\Lambda}_{n \setminus k})\end{bmatrix}\right) 
    \\&=\tr(f(\bm{\Lambda}_{n \setminus k})\bm{F}(\bm{I} + \bm{F}^\T \bm{F})^{-1}\bm{F}^\T f(\bm{\Lambda}_{n \setminus k}))
    \\&\leq
    \tr(f(\bm{\Lambda}_{n \setminus k})\bm{F}\bm{F}^\T f(\bm{\Lambda}_{n \setminus k}))
    =
    \|f(\bm{\Lambda}_{n \setminus k}) \bm{F}\|_{\F}^2 
    \\&\leq \|f(\bm{\Lambda}_{n \setminus k})\|_2^2 \|p(\bm{\Lambda}_k)^{-1}\|_2^2 \|p(\bm{\Lambda}_{n \setminus k}) \bm{\Omega}_{n \setminus k} \bm{\Omega}_k^{\dagger}\|_{\F}^2 
    \\&\leq 
    \|p(\bm{\Lambda}_{n \setminus k})\bm{\Omega}_{n \setminus k} \bm{\Omega}_k^{\dagger}\|_{\F}^2\max\limits_{i=1,\ldots,k} \left|\frac{f(\lambda_i)}{p(\lambda_i)}\right|^2.\numberthis \label{eq:third_term}
\end{align*}
Inserting the bounds \eqref{eq:first_term}, \eqref{eq:second_term}, and \eqref{eq:third_term} into \eqref{eq:two_terms} and optimizing over $\mathbb{P}_{s-1}$ yields the desired inequality. 
\end{proof}

\subsection{Probabilistic bounds}\label{section:probabilistic}
With the structural bound available, we are ready to derive probabilistic bounds for $\|f(\bm{A})- \bm{Q}_s \bm{Q}_s^\T f(\bm{A}) \bm{Q}_s \bm{Q}_s^\T\|_{\F}$ and $\|f(\bm{A})- \bm{Q}_s \llbracket\bm{Q}_s^\T f(\bm{A}) \bm{Q}_s\rrbracket_k \bm{Q}_s^\T\|_{\F}$. Note that by \Cref{lemma:structural} it is sufficient to derive a probabilistic bound for $\mathcal{E}_{\bm{\Omega}}(s;f)$ defined in \cref{eqn:min_ratio_omega}.

We will bound $\mathcal{E}_{\bm{\Omega}}(s;f)$ in terms of a deterministic quantity
\begin{equation}
\label{eqn:min_ratio}
    \mathcal{E}(s;f) =
    \min\limits_{p \in \mathbb{P}_{s-1}}\left[\|p(\bm{\Lambda}_{n \setminus k}) \|_{\F}^2\max\limits_{i=1,\ldots,k} \left|\frac{f(\lambda_i)}{p(\lambda_i)}\right|^2\right] ,
\end{equation}
which again quantifies how large a polynomial can be (relative to $f$) on the eigenvalues $\lambda_1, \ldots, \lambda_k$ and small on the remaining eigenvalues. However, $\mathcal{E}(s;f)$ does not depend on the randomness used by the algorithm.
\begin{lemma}\label{lemma:probabilistic}
    If $\rev{\bm{\Omega} \in \mathbb{R}^{n \times \ell}}$ is a standard Gaussian matrix, $\bm{\Omega}_k$ and $\bm{\Omega}_{n \setminus k}$ are as defined as in \eqref{eq:omega_partition}, and $\mathcal{E}_{\bm{\Omega}}(s;f)$ and $\mathcal{E}(s;f)$ are as defined in \cref{eqn:min_ratio_omega,eqn:min_ratio}, then
    \begin{enumerate}[(i)]
        \item \rev{for any $\delta \in (0,1)$ and $\ell \geq k$, with probability at least $1-\delta$,
        \begin{align*}
            \mathcal{E}_{\bm{\Omega}}(s;f)\leq C_{\delta,k,\ell} \mathcal{E}(s;f),
        \end{align*}\label{item:tailbound}
        where $C_{\delta,k,\ell} = 2e\left(2\log(2/\delta) + 1\right)\left(\frac{2\sqrt{\pi k}}{\delta}\right)^{\frac{2}{\ell-k+1}}\frac{k}{\ell - k + 1}$;}
        \item if $\ell -k \geq 2$ we have \begin{align*}
            \mathbb{E}[\mathcal{E}_{\bm{\Omega}}(s;f)] 
            \leq \frac{k}{\ell-k-1} \mathcal{E}(s;f).
        \end{align*}\label{item:expectationbound}
    \end{enumerate}
\end{lemma}
\begin{proof}
\textit{(\ref*{item:tailbound})}: For any polynomial $p \in \mathbb{P}_{s-1}$ by \cite[Proposition 8.6]{tropp2023randomized} we have with probability at least $1-e^{-(u-2)/4} - \sqrt{\pi k} \left(\frac{t}{e}\right)^{-(\ell -k + 1)/2}$
\begin{equation*}
    \left[\|p(\bm{\Lambda}_{n \setminus k}) \bm{\Omega}_{n \setminus k} \bm{\Omega}_k^{\dagger}\|_{\F}^2\max\limits_{i=1,\ldots,k} \left|\frac{f(\lambda_i)}{p(\lambda_i)}\right|^2\right] \leq \frac{ut k}{\ell - k + 1}\left[\|p(\bm{\Lambda}_{n \setminus k})\|_{\F}^2\max\limits_{i=1,\ldots,k} \left|\frac{f(\lambda_i)}{p(\lambda_i)}\right|^2 \right].
\end{equation*}
The inequality is respected if we minimize both sides over all polynomials. \rev{Solving $e^{(u-2)/2} \leq \delta/2$ and $\sqrt{\pi k} \left(\frac{t}{e}\right)^{-(\ell -k + 1)/2} \leq \delta/2$ for $u$ and $t$ yields the desired result.}

\textit{(\ref*{item:expectationbound})}: This is proven in an identical fashion utilizing the expectation bound in \cite[Proposition 8.6]{tropp2023randomized}.
\end{proof}

\subsection{Error bounds for Krylov aware low-rank approximation}\label{section:krylow}

With \Cref{theorem:robust}, \Cref{lemma:structural}, and \Cref{lemma:probabilistic} we can now derive a probabilistic error bound for $\ALG_k(s,r;f)$, the output of \Cref{alg:krylow}. By an almost identical argument one can obtain a similar bound for $\ALG(s,r;f)$, but we omit the details.


\begin{theorem}\label{theorem:krylov_aware}
Consider $\bm{A} \in \mathbb{R}^{n \times n}$ as defined in \eqref{eq:A} with smallest and largest eigenvalues $\lambda_{\min}$ and $\lambda_{\max}$ respectively.
Then, with $\mathcal{E}(s;f)$ 
as defined in \cref{eqn:min_ratio} 
\begin{enumerate}[(i)]
        \item \rev{with probability at least $1-\delta$,
        \begin{align*}
        \|f(\bm{A}) - \ALG_k(s,r;f) \|_{\F} \leq &4\sqrt{\ell s}  \inf\limits_{p \in \mathbb{P}_{2r+1}}\|f(x)-p(x)\|_{L^{\infty}([\lambda_{\min},\lambda_{\max}])}+\\
        &\sqrt{\|f(\bm{\Lambda}_{n \setminus k })\|_{\F}^2 +  5C_{\delta,k,\ell} \mathcal{E}(s;f)},
        \end{align*}\label{item:krylov_aware_tailbound}
        where $C_{\delta,k,\ell}$ is defined in \Cref{lemma:probabilistic};}
        \item if $\ell -k \geq 2$ that
        \begin{align*}
        \mathbb{E}\|f(\bm{A}) - \ALG_k(s,r;f) \|_{\F} \leq &4\sqrt{\ell s}  \inf\limits_{p \in \mathbb{P}_{2r+1}}\|f(x)-p(x)\|_{L^{\infty}([\lambda_{\min},\lambda_{\max}])}+\\
        &\sqrt{\|f(\bm{\Lambda}_{n \setminus k })\|_{\F}^2 +  \frac{5k}{\ell - k - 1} \mathcal{E}(s;f)}.        \end{align*}\label{item:krylov_aware_expectationbound}
    \end{enumerate}
\end{theorem}

\begin{proof}
\textit{(\ref*{item:krylov_aware_tailbound})}: By applying \Cref{theorem:robust}, \Cref{lemma:structural}, and \Cref{lemma:2_times_polynomial_approx} we obtain the following structural bound 
\begin{align}
\begin{split}
        \|f(\bm{A}) - \bm{Q}_s \llbracket f(\bm{T}_q)_{1:d_{s,\ell},1:d_{s,\ell}}\rrbracket_{k} \bm{Q}_s^\T\|_{\F} \leq &4\sqrt{\ell s}  \inf\limits_{p \in \mathbb{P}_{2r+1}}\|f(x)-p(x)\|_{L^{\infty}([\lambda_{\min},\lambda_{\max}])} +\\
        &\sqrt{\|f(\bm{\Lambda}_{n \setminus k })\|_{\F}^2 +  5 \mathcal{E}_{\bm{\Omega}}(s;f)}.
        \end{split}\label{eq:krylov_aware_structural}
\end{align}
Applying \Cref{lemma:probabilistic} yields \textit{(\ref*{item:krylov_aware_tailbound})} directly. 
\textit{(\ref*{item:krylov_aware_expectationbound})} follows from \Cref{lemma:probabilistic} and an application of Jensen's inequality.
\end{proof}

We conclude this section by commenting on the three terms appearing in the bounds in \Cref{theorem:krylov_aware}. 
The dependence on $4\sqrt{\ell s}  \inf\limits_{p \in \mathbb{P}_{2r+1}}\|f(x)-p(x)\|_{L^{\infty}([\lambda_{\min},\lambda_{\max}])}$ captures the inherent cost of approximating a quadratic form with $f(\bm{A})$ using the Lanczos method, as in done in Line 3 of \Cref{alg:krylow}. A similar term would also arise in an analysis of the Randomized SVD when matvecs are computed approximately using a Krylov subspace method.
The $\|f(\bm{\Lambda}_{n \setminus k})\|_{\F}$ term is due to the fact that our error can never be below the optimal rank $k$ approximation error. Finally, $\mathcal{E}(s;f)$ tells us that $\bm{Q}_s$ is a good orthonormal basis for low-rank approximation if there is a polynomial of degree at most $s-1$ that is large on the eigenvalues $\lambda_1,\ldots,\lambda_k$ (which correspond to the top subspace of $f(\bm{A})$) and is small on the eigenvalues $\lambda_{k+1},\ldots,\lambda_n$. 

As discussed in \Cref{section:structural}, one possible choice of polynomial would be an approximation to $f$, in which case $\mathcal{E}(s;f)$ would be close to $\|f(\bm{\Lambda}_{n \setminus k})\|_{\F}$. In this case, ignoring the polynomial approximation term, we would recover a bound almost identical to the standard Randomized SVD error bound. In particular, when matvecs with $f(\bm{A})$ are implemented exactly, Randomized SVD can be shown to have expected error \cite{rsvd}:
\begin{align*}
\sqrt{\|f(\bm{\Lambda}_{n \setminus k})\|_{\F}^2 + \frac{k}{\ell-k-1}\|f(\bm{\Lambda}_{n \setminus k})\|_{\F}^2}.
\end{align*}
\rev{When $f$ is a polynomial of degree at most $\max\{2r+1,s-1\}$, we can recover classical bounds for the randomized SVD.} However, other possibly polynomials can also be chosen, so we might have $\mathcal{E}(s;f)  \ll \|f(\bm{\Lambda}_{n \setminus k})\|_{\F}$. In the next section, we give an example involving the matrix exponential to better illustrate why this is often the case.

\subsection{Example bounds for the matrix exponential}\label{section:exponential}
\rev{
In order to obtain explicit bounds for the approximation quality of \Cref{alg:krylow} one must obtain an explicit upper bound for $\mathcal{E}(s;f)$. This can be done, for example, by choosing a specific polynomial $p$ to obtain
\begin{equation*}
    \mathcal{E}(s;f) \leq \|p(\bm{\Lambda}_{n \setminus k})\|_\F^2 \max\limits_{i=1,\ldots,k} \left|\frac{f(\lambda_i)}{p(\lambda_i)}\right|.
\end{equation*}
For monotone functions, a natural choice for $p$ is a scaled a shifted Chebyshev polynomial, as suggested in \cite{MM15, tropp2023randomized}; this idea is also standard in the context of eigenvalue approximation \cite{shmuel,paige1971computation,saad1980rates,underwood1975iterative}.}

One \rev{monotonic} matrix function for which it is often desirable to obtain a low-rank approximation is the matrix exponential, $\exp(\bm{A})$. \rev{As highlighted in \Cref{section:motivation}, this} task arises in tasks ranging from network analysis \cite{hpp}, to quantum thermodynamics \cite{chen_hallman_23,epperly2023xtrace}, to solving PDEs \cite{persson_kressner_23}. \rev{We emphasize that in many cases $\exp(\bm{A})$ admits an accurate low-rank approximation, even if $\bm{A}$ does not. Furthermore, at first glance one might think that one must require that the eigenvalues of $\bm{A}$ be negative for $\exp(\bm{A})$ to admit an accurate low-rank approximation. However, this is not the case. For any symmetric matrix $\bm{A}$ with eigenvalues $\lambda_1 \geq \lambda_2 \geq \cdots \geq \lambda_n$, consider the shifted matrix $\bm{A} - \lambda_1 \bm{I}$, whose eigenvalues are all non-positive. Since $\exp(\bm{A} - \lambda_1 \bm{I}) = \exp(-\lambda_1) \exp(\bm{A})$, the matrix exponentials differ only by a scalar multiple. Consequently, any relative low-rank approximation error remains unchanged under shifts. Moreover, \Cref{alg:krylow} is unaffected by such a shift since $\mathcal{K}_q(\bm{A},\bm{\Omega}) = \mathcal{K}_q(\bm{A} - \lambda_1 \bm{I},\bm{\Omega})$.} 

In this section, we do a deeper dive into our main bounds for Krylov aware low-rank approximation when applied to the matrix exponential to better understand how they compare to bounds obtainable using the more naive approach of directly combining the randomized SVD with a black-box matrix-vector product approximation algorithm for $\exp(\bm{A})$.  For conciseness of exposition, we focus on expectation bounds, and defer details of proofs to \Cref{section:appendix}.

We first demonstrates that, as discussed in the previous section, by choosing $p$ in $\mathcal{E}(s;\exp(x))$ to be a polynomial approximation to $\exp(x)$, we can recover the bounds of the randomized SVD \cite[Theorem 10.5]{rsvd}, up to small factors accounting for the fact that $\exp(x)$ is not a polynomial. 



\begin{corollary}\label{theorem:rsvd_like_bound}
Consider the setting of \Cref{theorem:krylov_aware} with $f(x) = \exp(x)$. Then, if $\ell-k\geq 2$ and $\gamma_{1,n} := \lambda_{\max} - \lambda_{\min} = \lambda_1-\lambda_n$, and $s \geq e\gamma_{1,n}$ we have
\begin{align*}
     &\mathbb{E}\|\exp(\bm{A}) - \ALG_k(s,r;\exp(x))\|_{\F} 
     \\&\hspace{2em}
     \leq  \frac{\sqrt{\ell s}\gamma_{1,n}^{2r+2}}{2^{4r+1}(2r+2)!}\|\exp(\bm{A})\|_2 + \sqrt{1 +  \frac{1}{(1-\frac{\gamma_{1,n}^s}{s!})^2}\frac{5k}{\ell - k - 1}} \|\exp(\bm{\Lambda}_{n \setminus k })\|_{\F}.
\end{align*}
\end{corollary}
Note that as $\frac{\sqrt{\ell s}\gamma_{1,n}^{2r+2}}{2^{4r+1}(2r+2)!}\|\exp(\bm{A})\|_2 \to 0$ and $\frac{\gamma_{1,n}^s}{s!} \to 0$, we recover the classical bound for the randomized SVD \cite[Theorem 10.5]{rsvd}. Furthermore, these two terms converge \emph{superexponentially} to $0$ as $s,r \to +\infty$. However, due to the $\frac{5k}{\ell -k -1}$, the above bound only implies a subexponential convergence to the error of the optimal rank $k$ approximation error, $\|\exp(\bm{\Lambda}_{n \setminus k })\|_{\F}$, in terms of the number of required matrix-vector products with $\bm{A}$, which scales as $\ell (r+s)$.

However, as argued in \Cref{section:structural} there are other polynomials that can achieve a significantly tighter upper bound of $\mathcal{E}(s;\exp(x))$. For example, by choosing $p$ in $\mathcal{E}(s;\exp(x))$ to be a scaled and shifted Chebyshev polynomial we can show the output of \Cref{alg:krylow} converges \emph{exponentially} in $s$ to the optimal low-rank approximation of $\exp(\bm{A})$, even if $\ell$ stays fixed. 
 This demonstrates why \Cref{alg:krylow} is expected to return a better low-rank approximation compared to applying the randomized SVD immediately on $\exp(\bm{A})$. In particular, we have the following result.

\begin{corollary}\label{theorem:fast_convergence}
Consider the setting of \Cref{theorem:krylov_aware} with $f(x) = \exp(x)$. Define $\gamma_{i,j} = \lambda_{i} - \lambda_{j}$ to be the gap between the $i^{\text{th}}$ and $j^{\text{th}}$ largest eigenvalues. \rev{$\Gamma=\min\left\{1, 2 \frac{\gamma_{k,k+1}}{\gamma_{k+1,n}}\right\}$.} Then, if $\ell-k\geq 2$ and $s \geq 2 + \frac{\gamma_{1,k}}{\log\left(\frac{\gamma_{1,n}}{\gamma_{k,n}}\right)}$
\begin{align*}
     &\mathbb{E}\|\exp(\bm{A}) - \ALG_k(s,r;\exp(x))\|_{\F} 
     \leq  \frac{\sqrt{\ell s}\gamma_{1,n}^{2r+2}}{2^{4r+1}(2r+2)!}\|\exp(\bm{A})\|_2 
     \\&\hspace{5em}+ \sqrt{1 +  \rev{2^{-2(s-2)\sqrt{\Gamma} + 3\gamma_{k,n}}}\frac{80k}{\ell - k - 1}} \|\exp(\bm{\Lambda}_{n \setminus k })\|_{\F}.
\end{align*}
\rev{Consequently, if $s \geq 2 + \frac{\gamma_{1,k}}{\log\left(\frac{\gamma_{1,n}}{\gamma_{k,n}}\right)}$ and
\begin{align*}
    \ell &= 2k+1 = O(k),\\
    s &= O\left(\frac{\log(1/\varepsilon) + \gamma_{k,n}}{\sqrt{\Gamma}}\right),\\
    r &= O\left(\gamma_{1,n} +\log(1/\tau)+ \log\left(\frac{k(\log(1/\varepsilon) + \gamma_{k,n})}{\Gamma}\right)\right),
\end{align*}
for $\varepsilon, \tau \in (0,1)$, then we have
\begin{equation*}
    \mathbb{E}\|\exp(\bm{A}) - \ALG_k(s,r;\exp(x))\|_{\F} 
     \leq \tau\|\exp(\bm{A})\|_2 + (1+\varepsilon)\|\exp(\bm{\Lambda}_{n \setminus k})\|_\F.
\end{equation*}}
\end{corollary}
\rev{We believe that one can derive bounds that are independent of the eigenvalue gap $\Gamma$, reminiscent of the bounds in \cite{MM15}. However, such bounds would likely require very different proof techniques.}
\section{\rev{Single-vector Krylov-aware low-rank approximation}}\label{section:singlevector}
\rev{In this section we consider a variant of \Cref{alg:krylow} that uses a single vector $\bm{\omega} \in \mathbb{R}^n$ as a starting block, rather than a block of vectors $\bm{\Omega}$. The resulting algorithm is given in \Cref{alg:krylow_single}, and requires at most $k+s+r$ matvecs with $\bm{A}$ to compute an approximation to $f(\bm{A})$. This compares favourably to \Cref{alg:krylow}, which requires $k(s+r)$ matvecs. The overall computational cost of \Cref{alg:krylow_single} is $k+s+r$ matvecs and $O(n(k+s+r)^2)$ additional operations. 
\begin{algorithm}
\caption{Single-vector Krylov-aware low-rank approximation}
\label{alg:krylow_single}
\textbf{input:} Symmetric $\bm{A} \in \mathbb{R}^{n \times n}$. Rank $k$. Matrix function $f: \mathbb{R} \to \mathbb{R}$. Degrees $s$ and $r$. \\
\textbf{output:} Low-rank approximation to $f(\bm{A})$: $\bm{Q}_{k+s} \bm{X} \bm{Q}_{k+s}^\T$ or $\bm{Q}_{k+s} \llbracket \bm{X} \rrbracket_k  \bm{Q}_{k+s}^\T$
\begin{algorithmic}[1]
    \State Sample a standard Gaussian random vector $\bm{\omega}$.
    \State Run \cref{alg:block_lanczos} for $q=k+s+r$ iterations to obtain an orthonormal basis $\bm{Q}_{k+s}$ for $\mathcal{K}_{k+s}(\bm{A},\bm{\omega})$ and a tridiagonal matrix $\bm{T}_q$.\label{line:single_basis}
    \State Compute $\bm{X} = f(\bm{T}_q)_{1:d_{k+s,1},1:d_{k+s,1}}$ where $d_{k+s,1} = \dim(\mathcal{K}_{k+s}(\bm{A},\bm{\omega}))$. \Comment{$\approx \bm{Q}_{k+s}^\T f(\bm{A}) \bm{Q}_{k+s}$} 
    \State \textbf{return} $\widehat{\textsf{\textup{ALG}}}(s,r;f) = \bm{Q}_{k+s} \bm{X}  \bm{Q}_{k+s}^\T$ or $\widehat{\textsf{\textup{ALG}}}_k(s,r;f) = \bm{Q}_{k+s} \llbracket \bm{X} \rrbracket_k  \bm{Q}_{k+s}^\T$.\label{line:single_X}
\end{algorithmic}
\end{algorithm}}

\rev{Many classical eigenvalue methods use Krylov subspace methods with small starting blocks, often as small as $\ell = 1$ \cite{shmuel,paige1971computation}. For instance, MATLAB's \texttt{eigs} is based on a Krylov subspace initialized with a (random) starting block of size $\ell = 1$,\footnote{For reproducibility, MATLAB uses private random number stream.} even when several eigenvalues are sought. In contrast, Krylov subspace methods for low-rank approximation have historically been analyzed for large block-sizes, i.e., $\ell \geq k$ \cite{MM15,tropp2023randomized,drineasipsenkontopoulouismail,rsvd,gu_subspace}. The single-vector Lanczos method for low-rank approximation was recently analyzed in \cite{MeyerMuscoMusco:2024}. Empirically, the single-vector Lanczos method often produces more accurate low-rank approximations using the same number of matvecs, though the block-Lanczos method comes with other advantages, such as being easier to parallelize. Recent work has also studied intermediate block-sizes $1 \leq \ell \leq k$ in the context of low-rank approximation, which can balance accuracy and efficiency \cite{moderateblocksize}. In what follows, we will focus on the case when the block-size is $\ell = 1$, but we suspect that the analysis can be extended to any block-size.}

\rev{A key observation made in \cite{MeyerMuscoMusco:2024} is that the single-vector Lanczos method can be viewed as a block-Lanczos method with a highly structured starting block. Specifically, if $\bm{\omega} \in \mathbb{R}^{n}$ is a starting vector, then by \Cref{lemma:krylovkrylov} we have
\begin{equation*}
    \mathcal{K}_{s+k}(\bm{A}, \bm{\omega}) = \mathcal{K}_{s}(\bm{A}, \widehat{\bm{\Omega}}) \quad \text{ where } \quad \widehat{\bm{\Omega}} := \begin{bmatrix} \bm{\omega} & \bm{A} \bm{\omega} & \cdots & \bm{A}^{k-1} \bm{\omega} \end{bmatrix}.
\end{equation*}
The benefit of this observation was that the existing analysis for the block-Lanczos algorithm for low-rank approximation could be used to analyze the single-vector Lanczos method for low-rank approximation, provided one could analyze the behavior of the structured starting block $\widehat{\bm{\Omega}}$. To see this, note that running  \Cref{alg:krylow_single} is equivalent to running \Cref{alg:krylow} with $\widehat{\bm{\Omega}}$ instead of $\bm{\Omega}$. Hence, if we define 
\begin{equation*}
    \widehat{\bm{\Omega}}_k := \bm{U}_k^\T \widehat{\bm{\Omega}}, \quad \widehat{\bm{\Omega}}_{n \setminus k} := \bm{U}_{n \setminus k}^\T \widehat{\bm{\Omega}},
\end{equation*}
then, assuming $\rank(\widehat{\bm{\Omega}}_k) = k$ and applying \Cref{theorem:robust}, \Cref{lemma:structural},\footnote{Recall that \Cref{lemma:structural} does not assume anything about the randomness of the starting block.} and \Cref{lemma:2_times_polynomial_approx} we obtain the following structural bound
\begin{align}
\begin{split}
        \|f(\bm{A}) -\widehat{\textsf{\textup{ALG}}}_k(s,r;f)\|_{\F} \leq &4\sqrt{k+s}  \inf\limits_{p \in \mathbb{P}_{2r+1}}\|f(x)-p(x)\|_{L^{\infty}([\lambda_{\min},\lambda_{\max}])} +\\
        &\sqrt{\|f(\bm{\Lambda}_{n \setminus k })\|_{\F}^2 +  5 \min\limits_{p \in \mathbb{P}_{s-1}}\left[\|p(\bm{\Lambda}_{n \setminus k}) \widehat{\bm{\Omega}}_{n \setminus k} \widehat{\bm{\Omega}}_k^{-1}\|_{\F}^2\max\limits_{i=1,\ldots,k} \left|\frac{f(\lambda_i)}{p(\lambda_i)}\right|^2\right]}.
        \end{split}\label{eq:krylov_aware_structural_single}
\end{align}
Bounds for $\widehat{\textsf{\textup{ALG}}}(s,r;f)$ can be obtained analogously. To obtain a probabilistic bound, it remains to control $\|p(\bm{\Lambda}_{n \setminus k}) \widehat{\bm{\Omega}}_{n \setminus k} \widehat{\bm{\Omega}}_k^{-1}\|_{\F}^2$ for a fixed polynomial $p$, which can be done using the results from \cite{MeyerMuscoMusco:2024}. In particular, we have the following result, whose proof is deferred to \Cref{section:appendixB}.\footnote{Unfortunately, the proof technique cannot be used to prove expectation bounds, since the proof relies on bounding the inverse of a $\chi_1^2$ random variable, whose expectation does not exist.}}
\rev{\begin{theorem}\label{theorem:singlevector}
    Consider $\bm{A} \in \mathbb{R}^{n \times n}$ as defined in \eqref{eq:A} with smallest and largest eigenvalues $\lambda_{\min}$ and $\lambda_{\max}$ respectively. Define $\gamma_{\min} = \min\limits_{\substack{i,j = 1,\ldots k \\ i \neq j}}\frac{|\lambda_j-\lambda_i|}{|\lambda_{\max} - \lambda_{\min}|}$. Then, with $\mathcal{E}(s;f)$ as defined in \cref{eqn:min_ratio} the following inequality holds with probability at least $1-\delta$
     \begin{align*}
        \|f(\bm{A}) - \widehat{\ALG_k}(s,r;f) \|_{\F} \leq &4\sqrt{k+s}  \inf\limits_{p \in \mathbb{P}_{2r+1}}\|f(x)-p(x)\|_{L^{\infty}([\lambda_{\min},\lambda_{\max}])}+\\
        &\sqrt{\|f(\bm{\Lambda}_{n \setminus k })\|_{\F}^2 +  \widehat{C}_{\delta}\frac{k^4}{\gamma_{\min}^{2(k-1)}} \mathcal{E}(s;f)},
    \end{align*}
     where $\widehat{C}_{\delta} = \frac{2 \pi (1+6\log(2/\delta))}{\delta^2}$.
\end{theorem}}
\rev{We highlight the dependence on $\frac{1}{\gamma_{\min}^{2(k-1)}}$, which can be extremely large of $\bm{\Lambda}_k$ has tightly clustered eigenvalues. This term appears because of the dependence of the inverse of $\widehat{\bm{\Omega}}_k$ in \eqref{eq:krylov_aware_structural_single}, which becomes exponentially ill-conditioned in $k$ and depends severely on the separation of the eigenvalues in $\bm{\Lambda}_k$. In particular, if $\bm{A}$ has repeated eigenvalues, $\widehat{\bm{\Omega}}$ may become rank deficient. To illustrate, if $\bm{A} = \bm{I}$ then $\dim(\mathcal{K}_{s+k}(\bm{A}, \bm{\omega})) = 1$, independently of $s+k$. In particular, we have $\rank(\widehat{\bm{\Omega}}) = 1$ and $\gamma_{\min} = 0$. Hence, the single-vector Lanczos method may fail to produce non-trivial low-rank approximations in degenerate cases when $\bm{A}$ has repeated eigenvalues.}

\rev{However, as argued in \cite{MeyerMuscoMusco:2024}, provided $\gamma_{\min} \neq 0$, the dependence on $\gamma_{\min}$ is mild since $\mathcal{E}(s;f)$ usually decays rapidly in $s$. To illustrate, consider the case when $f(x) = \exp(x)$. In the proof of \Cref{theorem:fast_convergence} we showed that
\begin{equation*}
    \mathcal{E}(s;f) = O\left(2^{-2(s-2)\sqrt{\Gamma} + 3\gamma_{k,n}}\|\exp(\bm{\Lambda}_{n \setminus k})\|_\F^2\right).
\end{equation*}
Thus, to ensure $\widehat{C}_{\delta}\frac{k^4}{\gamma_{\min}^{2(k-1)}} \mathcal{E}(s;f) \leq \varepsilon \|\exp(\bm{\Lambda}_{n \setminus k})\|_\F^2$ it is sufficient set $$s = O\left(\frac{\log(1/\varepsilon) +\gamma_{k,n} + k \log(1/\gamma_{\min}) + \log(\widehat{C}_{\delta})}{\sqrt{\Gamma}}\right).$$
Hence, the dependence on $\frac{1}{\gamma_{\min}}$ only appears in a logarithmic factor. As we will see in \Cref{sec:experiments}, \Cref{alg:krylow_single} usually produce significantly more accurate low-rank approximations using a fixed computational cost. 
}

\section{Numerical experiments}
\label{sec:experiments}
\rev{In this section we compare the Krylov aware low-rank approximation (\Cref{alg:krylow}), \Cref{alg:rsvd} (assuming exact matvecs with $f(\bm{A})$) and \Cref{alg:rsvd_matfun} (inexact matvecs with $f(\bm{A})$).} All experiments have been performed in MATLAB (version \rev{2024b}) and scripts to reproduce the figures are available at \url{https://github.com/davpersson/Krylov_aware_LRA.git}.

\subsection{Test matrices}
We begin with outlining the test matrices and matrix functions used in our examples.

\subsubsection{Exponential integrator}\label{section:exponential_integrator}
The following example is taken from \cite{persson_kressner_23}. Consider the following parabolic differential equation
\begin{align*}
\begin{split}
    &u_t = \kappa \Delta u + \lambda u \text{ in } [0,1]^2 \times [0,2]\\
    &u(\cdot,0) = \theta \text{ in } [0,1]^2\\
    &u = 0 \text{ on } \Gamma_1\\
    & \frac{\partial u}{\partial \bm{n}} = 0 \text{ on } \Gamma_2
\end{split}
\end{align*}
for $\kappa, \lambda > 0$ and $\Gamma_2 = \{(x,1) \in \mathbb{R}^{2} : x \in [0,1] \}$ and $\Gamma_1 = \partial \mathcal{D} \setminus \Gamma_2$. By discretizing in space using finite differences on a $100 \times 100$ grid we obtain an ordinary differential equation of the form
\begin{align}
\begin{split}
    \dot{\bm{u}}(t) &= \bm{A} \bm{u}(t) \text{ for } t \geq 0,\\
    \bm{u}(0) &= \bm{\theta}, \label{eq:ode}
\end{split}
\end{align}
for symmetric matrix $\bm{A} \in \mathbb{R}^{9900 \times 9900}$. It is well known that the solution to \eqref{eq:ode} is given by $\bm{u}(t) = \exp(t\bm{A})\bm{\theta}$. Suppose that we want to compute the solution for $t \geq 1$. One can verify that 
\begin{equation*}
    \max\limits_{t \geq 1} \frac{\|\exp(t\bm{A}) - \llbracket\exp(t\bm{A})\rrbracket_{k}\|_{\F}}{\|\exp(t\bm{A})\|_{\F}} = \frac{\|\exp(\bm{A}) - \llbracket\exp(\bm{A})\rrbracket_{k}\|_{\F}}{\|\exp(\bm{A})\|_{\F}},
\end{equation*}
and it turns out that $\exp(\bm{A})$ admits a good rank $60$ approximation
\begin{equation*}
    \frac{\|\exp(\bm{A}) - \llbracket\exp(\bm{A})\rrbracket_{60}\|_{\F}}{\|\exp(\bm{A})\|_{\F}} \approx 4 \times 10^{-4}.
\end{equation*}
Hence, we can use \Cref{alg:krylow} to compute $\bm{Q}_s$ and $\bm{T}_q$ and use them to efficiently construct a rank $60$ approximation to $\exp(t\bm{A})$ for any $t$. 
\rev{Note that due to the well-known shift-invariance property of Krylov subspaces, as with most Krylov methods, an approximation can be obtained for many values of $t$ \emph{with little additional cost} \cite{Saad:1992,hochbruck_lubich,Hochbruck1998,Druskin1998}}.

In the experiments we set $\kappa = 0.01$ and $\lambda = 1$. 
\subsubsection{Estrada index}\label{section:estrada}
For an (undirected) graph with adjacency matrix $\bm{A}$ the Estrada index is defined as $\tr(\exp(\bm{A}))$. It is used to measure the degree of protein folding \cite{estrada}. One can estimate the Estrada index of a network by the Hutch++ algorithm or its variations \cite{chen_hallman_23,epperly2023xtrace,hpp,ahpp}, which requires computing a low-rank approximation of $\exp(\bm{A})$. Motivated by the numerical experiments in \cite{hpp} we let $\bm{A}$ be the adjacency matrix of Roget’s Thesaurus semantic graph \cite{roget}. 

\begin{figure}[ht]
\begin{subfigure}{.5\textwidth}
  \centering
  \includegraphics[width=.9\linewidth]{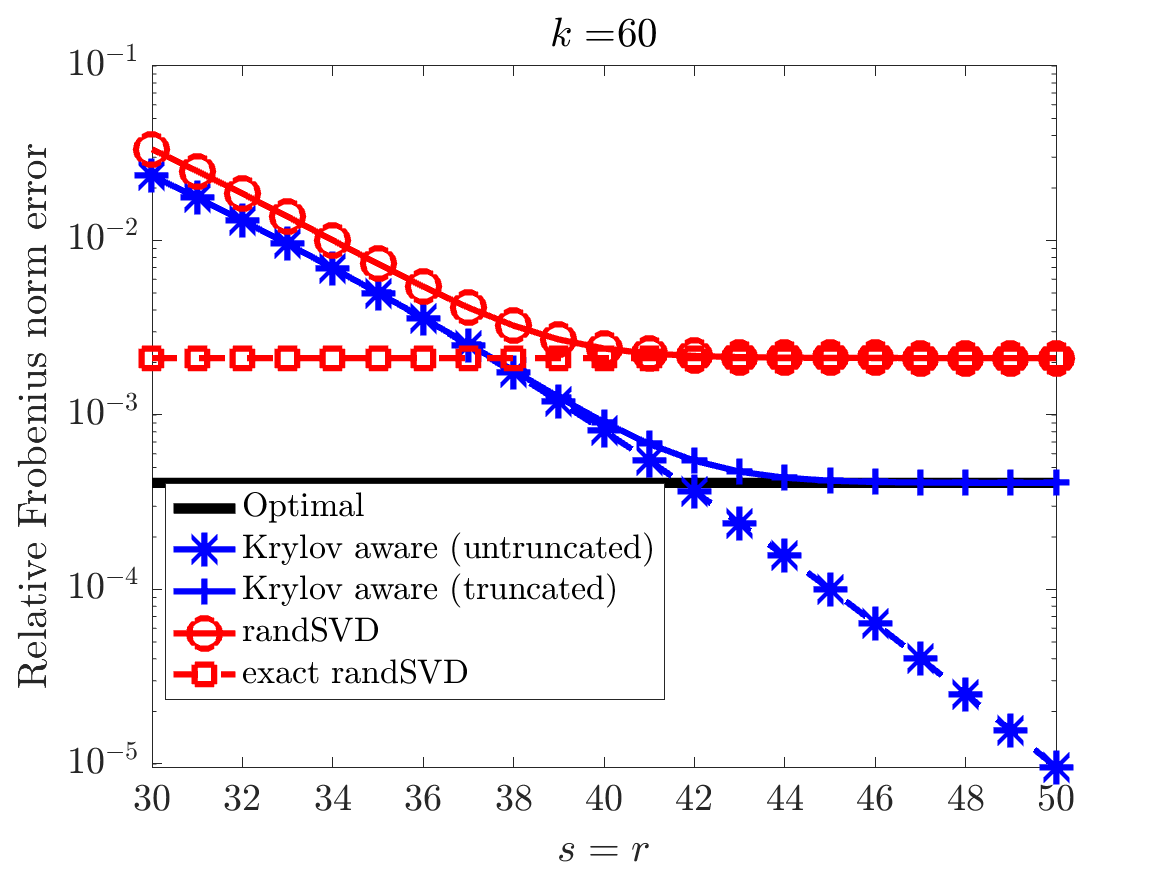}  
  \caption{Example from \Cref{section:exponential_integrator}}
\end{subfigure}
\begin{subfigure}{.5\textwidth}
  \centering
  \includegraphics[width=.9\linewidth]{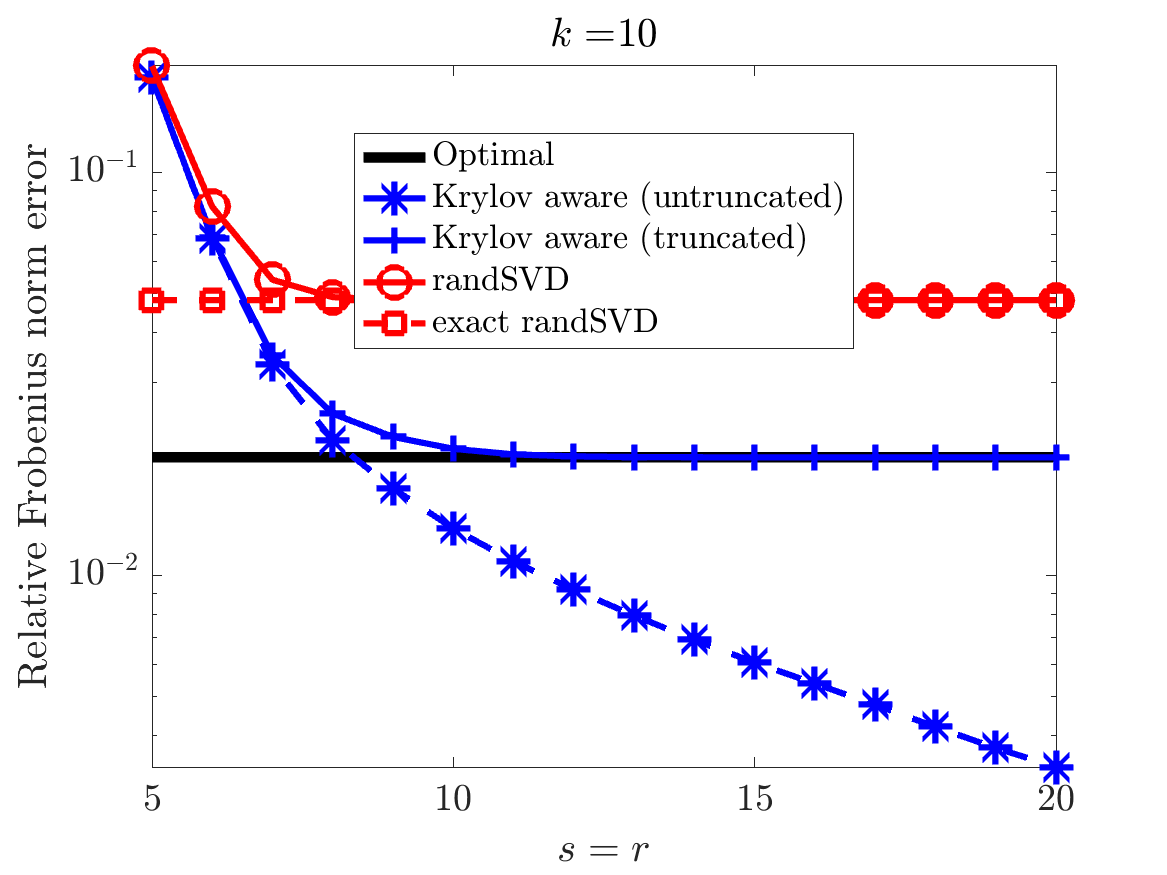} 
  \caption{Example from \Cref{section:estrada}}
\end{subfigure}
\begin{subfigure}{.5\textwidth}
  \centering
  \includegraphics[width=.9\linewidth]{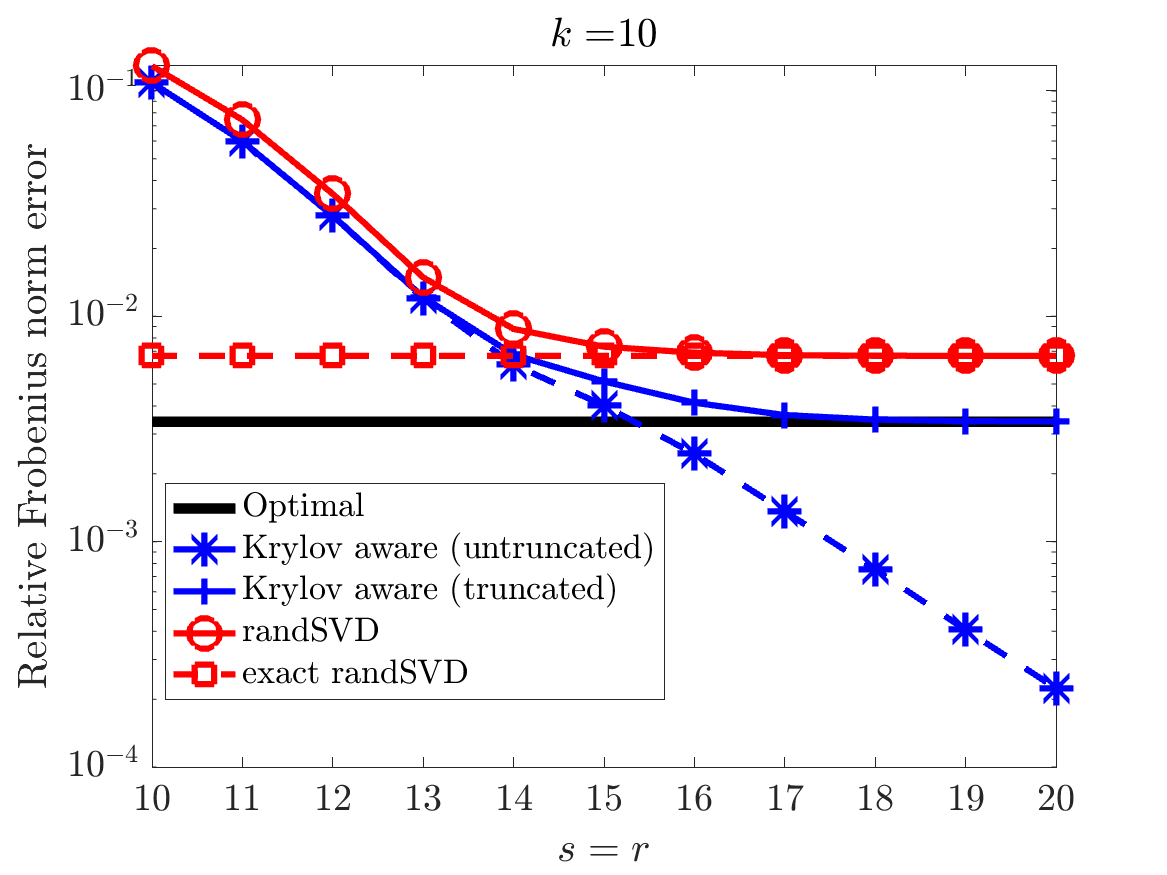}  
  \caption{Example from \Cref{section:quantum_spin}}
\end{subfigure}
\begin{subfigure}{.5\textwidth}
  \centering
  \includegraphics[width=.9\linewidth]{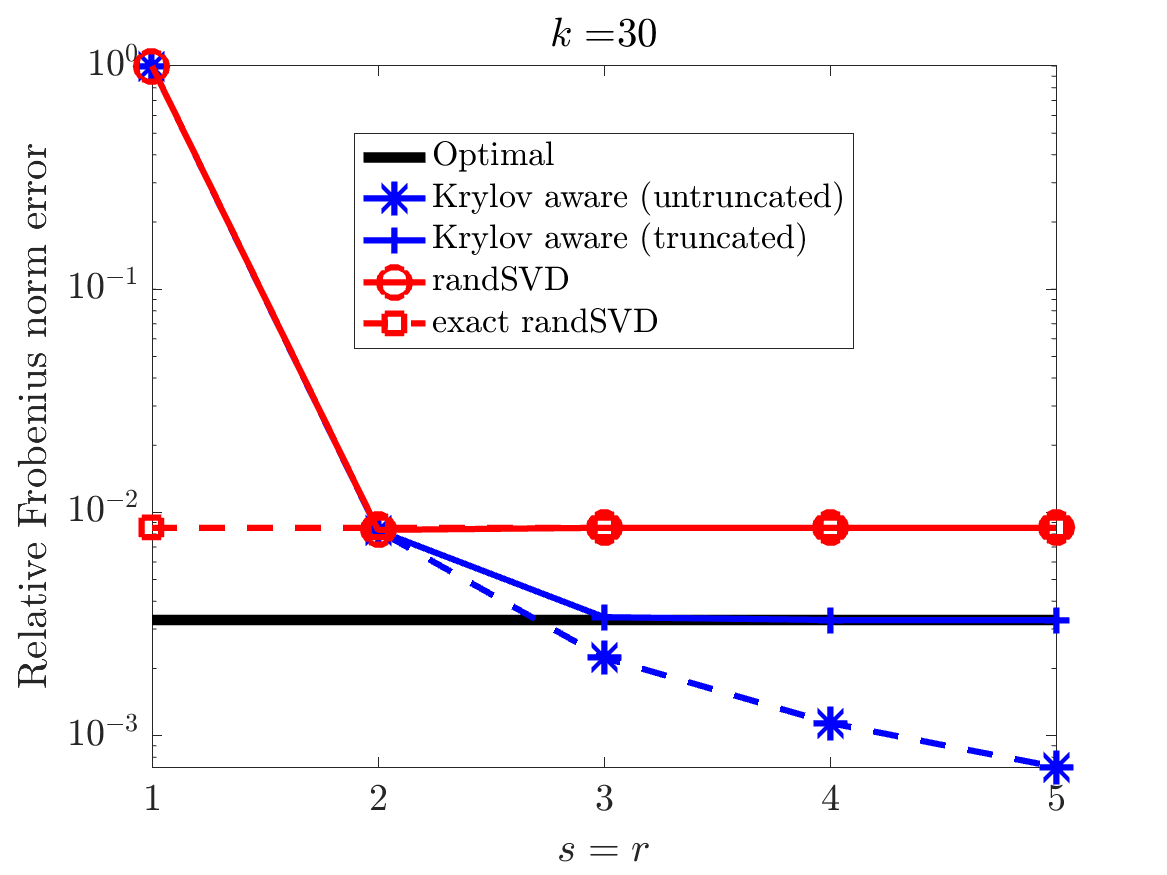}  
  \caption{Example from \Cref{section:synthetic_log}}
\end{subfigure}
\caption{Comparing relative error  \eqref{eq:relative_error} for the the approximations returned by \Cref{alg:krylow} without truncation (\rev{Krylov aware (untruncated)}), \Cref{alg:krylow} with truncation back to rank $k$ (\rev{Krylov aware (truncated)}), \Cref{alg:rsvd_matfun} \rev{(randSVD)}, and \Cref{alg:rsvd} \rev{(exact randSVD)}. The black line shows the optimal rank $k$ approximation relative Frobenius norm error. The rank parameter $k$ is visible as titles in the figures. In all experiments we set $\ell = k$. }
\label{fig:relative_errors}
\end{figure}

\subsubsection{Quantum spin system}\label{section:quantum_spin}
We use an example from \cite[Section 4.3]{epperly2023xtrace}, a similar example is found in \cite{chen_hallman_23}, in which we want to approximate $\exp(-\beta \bm{A})$ where 
\begin{equation*}
    \bm{A} = -\sum\limits_{i=1}^{N-1} \bm{Z}_i\bm{Z}_{i+1} -h\sum\limits_{i=1}^N \bm{X}_i \in \mathbb{R}^{n \times n},
\end{equation*}
where
\begin{equation*}
    \bm{X}_i = \bm{I}_{2^{i-1}} \otimes \bm{X} \otimes \bm{I}_{2^{N-i}}, \quad \bm{Z}_i = \bm{I}_{2^{i-1}} \otimes \bm{Z} \otimes \bm{I}_{2^{N-i}}
\end{equation*}
where $\bm{X}$ and $\bm{Z}$ are the Pauli operators
\begin{equation*}
    \bm{X} = \begin{bmatrix} 0 & 1 \\ 1 & 0 \end{bmatrix}, \quad \bm{Z} = \begin{bmatrix} 1 & 0 \\ 0 & -1 \end{bmatrix}.
\end{equation*}
Estimating the partition function $Z(\beta) = \tr(\exp(-\beta \bm{A}))$ is an important task in quantum mechanics \cite{pfeuty1970one}, which once again can benefit from computing a low-rank approximation of $\exp(-\beta \bm{A})$.

In the experiments we set $N = 14$ so that $n = 2^{14}$, $\beta = 0.3$, and $h = 10$.

\subsubsection{Synthetic example for the matrix logarithm}\label{section:synthetic_log}
We generate a symmetric matrix $\bm{A} \in \mathbb{R}^{5000 \times 5000}$ with eigenvalues $\lambda_i = \exp(\frac{1}{i^2})$ for $i = 1,\ldots,n$. We let $f(x) = \log(x)$ so that the eigenvalues of $f(\bm{A})$ are $f(\lambda_i) = \frac{1}{i^2}$ for $i = 1,\ldots,n$.

\begin{figure}[ht]
\begin{subfigure}{.5\textwidth}
  \centering
  \includegraphics[width=.9\linewidth]{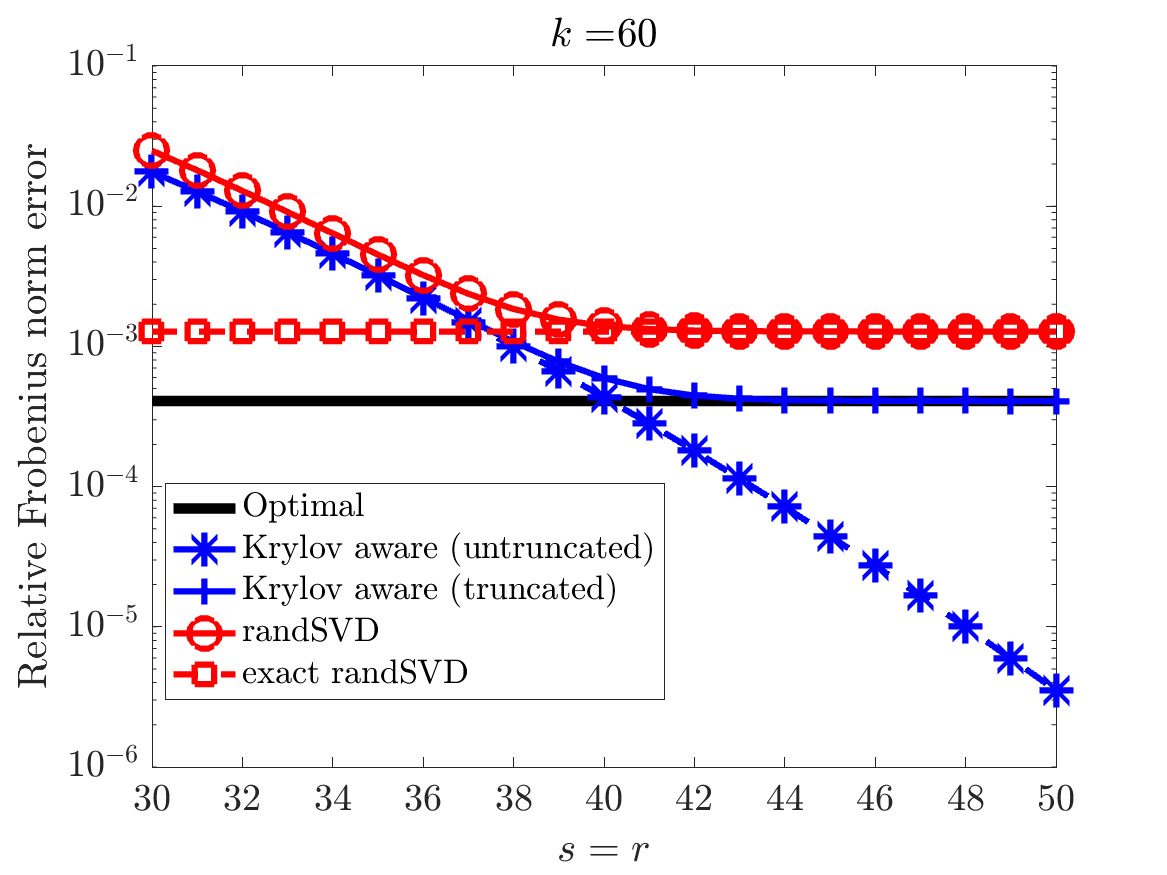}  
  \caption{Example from \Cref{section:exponential_integrator}}
\end{subfigure}
\begin{subfigure}{.5\textwidth}
  \centering
  \includegraphics[width=.9\linewidth]{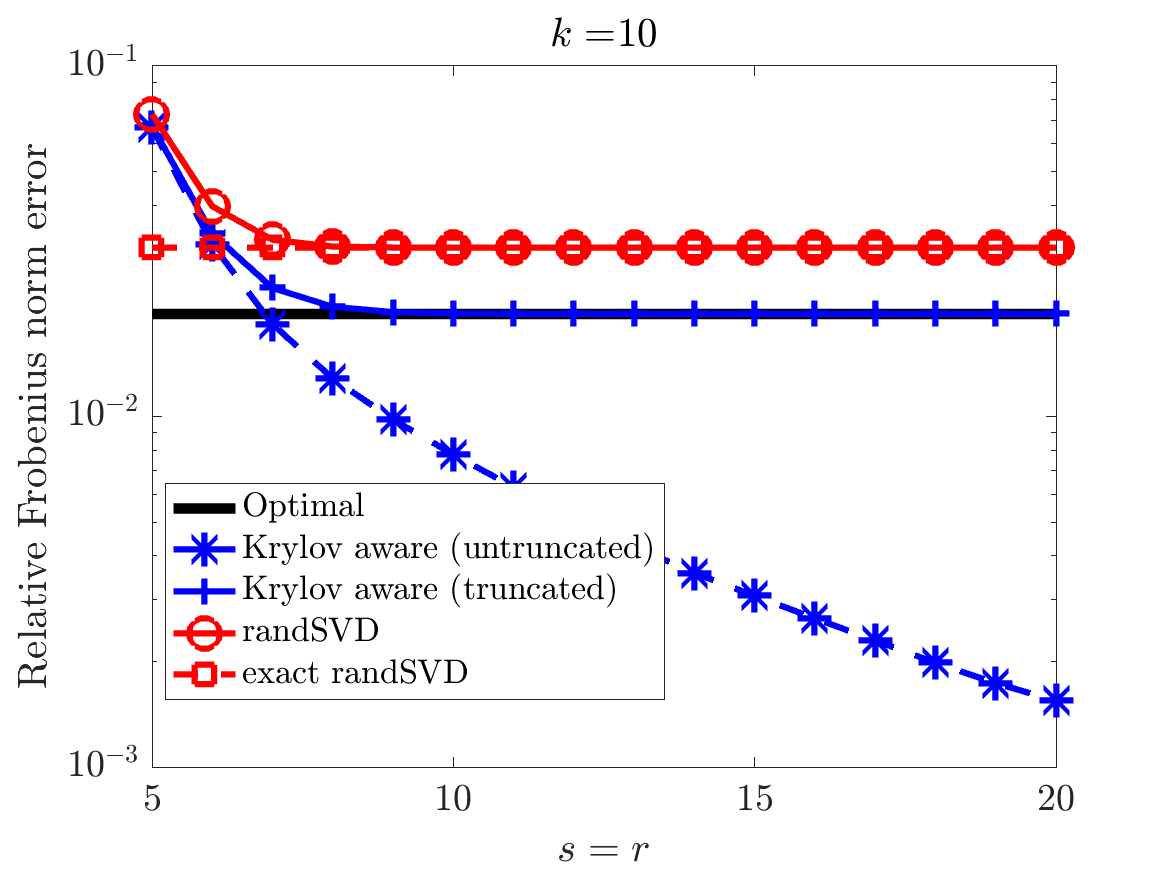} 
  \caption{Example from \Cref{section:estrada}}
\end{subfigure}
\begin{subfigure}{.5\textwidth}
  \centering
  \includegraphics[width=.9\linewidth]{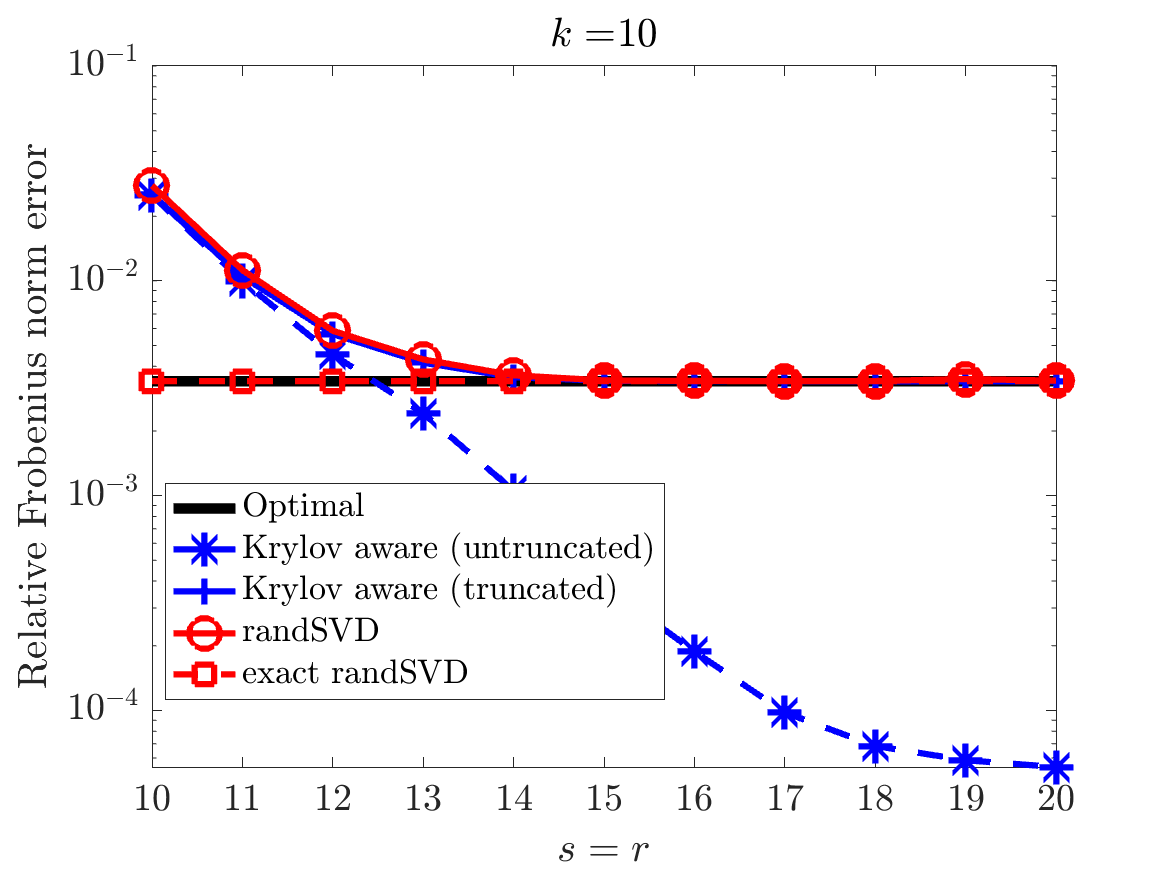}  
  \caption{Example from \Cref{section:quantum_spin}}
\end{subfigure}
\begin{subfigure}{.5\textwidth}
  \centering
  \includegraphics[width=.9\linewidth]{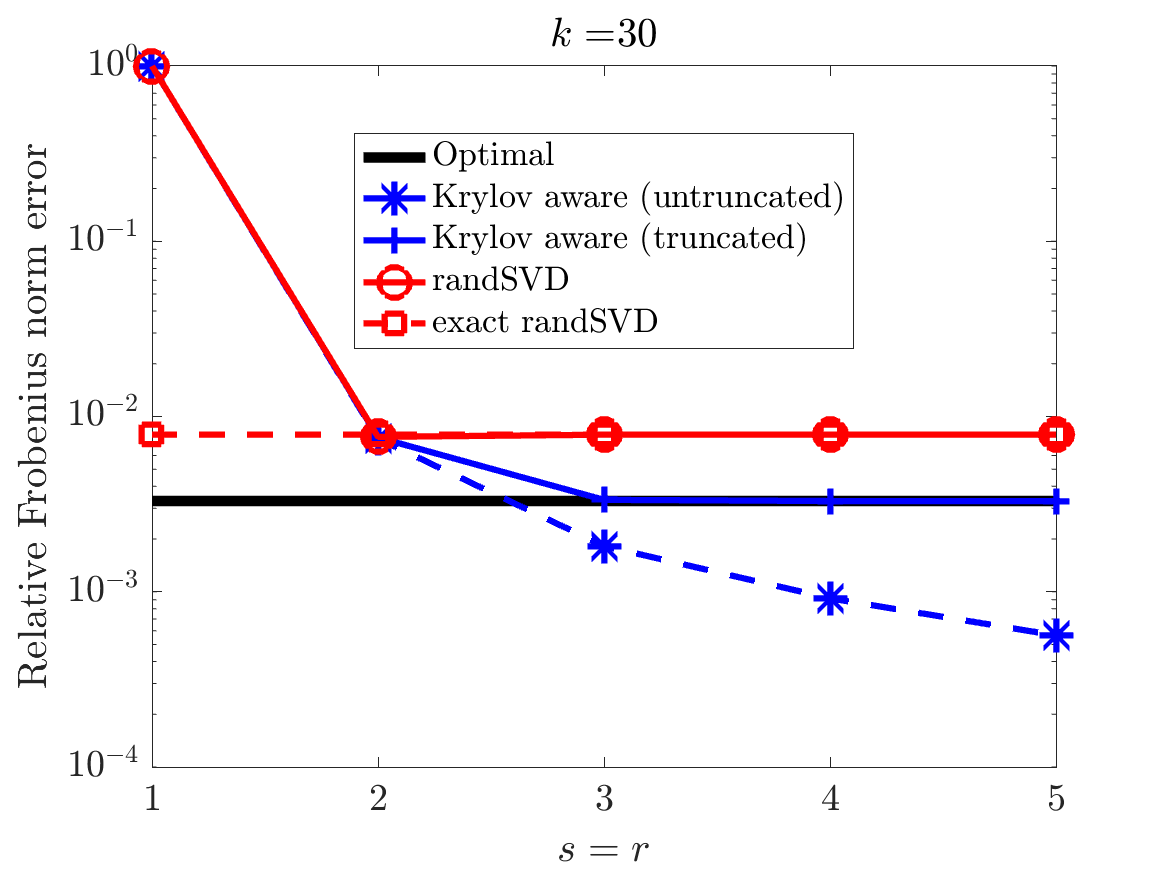}  
  \caption{Example from \Cref{section:synthetic_log}}
\end{subfigure}
\caption{Comparing relative error  \eqref{eq:relative_error} for the the approximations returned by \Cref{alg:krylow} without truncation (\rev{Krylov aware (untruncated)}), \Cref{alg:krylow} with truncation back to rank $k$ (\rev{Krylov aware (truncated)}), \Cref{alg:rsvd_matfun} \rev{(randSVD)}, and \Cref{alg:rsvd} \rev{(exact randSVD)}. The black line shows the optimal rank $k$ approximation relative Frobenius norm error. The rank parameter $k$ is visible as titles in the figures. In all experiments we set $\ell = k+5$. }
\label{fig:relative_errors2}
\end{figure}

\subsection{Comparing relative errors \rev{of block methods}}
In this section we compare the error \rev{of the approximations} returned by \Cref{alg:krylow},  \Cref{alg:rsvd}, and \Cref{alg:rsvd_matfun}. If $\bm{C}$ is a low-rank approximation returned by one of the algorithms then we compare the relative error
\begin{equation}\label{eq:relative_error}
    \frac{\|f(\bm{A}) - \bm{C}\|_{\F}}{\|f(\bm{A})\|_{\F}}.
\end{equation}
In all experiments we set the parameters in \Cref{alg:krylow} and \Cref{alg:rsvd_matfun} to be $\ell = k$ or $\ell = k + 5$ and $s = r$ so that the total number of matrix vector products with $\bm{A}$ is $2\ell s$. 
When we run \rev{\Cref{alg:rsvd}} we compute matvecs with $f(\bm{A})$ exactly, which cannot be done in practice. Hence, the results from this algorithm are only used as a reference for \Cref{alg:krylow} and \Cref{alg:rsvd_matfun}.
The results are presented in \Cref{fig:relative_errors} for $\ell = k$ and \Cref{fig:relative_errors2} for $\ell = k + 5$. All results confirm that \Cref{alg:krylow} returns a more accurate approximation than \Cref{alg:rsvd_matfun}, and can even be more accurate than \Cref{alg:rsvd}. \rev{These results also confirm the findings of Chen \& Hallman \cite{chen_hallman_23}, where it was shown that their Krylov aware version of Hutch++ yields more accurate approximations to $\tr(f(\bm{A}))$ compared to a naive implementation of Hutch++. As discussed in \Cref{section:introduction}, the reason for the effectiveness of their method is due to a more accurate low-rank approximation.} 

\subsection{\rev{Comparing relative errors of single vector methods}}
\rev{In this section we compare the relative error \eqref{eq:relative_error} of the approximations returned by \Cref{alg:krylow} and \Cref{alg:krylow_single}; that is, we compare the block and single-vector version of the Krylov-aware low-rank approximation.}

\rev{In \Cref{alg:krylow} we set the block-size to be $\ell = k$ and $s = r$. Consequently, \Cref{alg:krylow} performs at most $2sk$ matrix-vector products with $\bm{A}$ and returns the approximations $\ALG(s,s;f)$ and $\ALG_k(s,s;f)$. For comparison, we consider the output of \Cref{alg:krylow_single}, namely $\widehat{\textsf{\textup{ALG}}}((s-1)k,sk;f)$ and $\widehat{\textsf{\textup{ALG}}}_k((s-1)k,sk;f)$, which also require at most $2sk$ matrix-vector products with $\bm{A}$. The results are presented in \Cref{fig:relative_errors_svk}. These result, shown in \Cref{fig:relative_errors_svk},  indicate that with the same matrix-vector products with $\bm{A}$, the single-vector version \Cref{alg:krylow_single} yields a more efficient approximation to $f(\bm{A})$ than \Cref{alg:krylow}.}

\begin{figure}[ht]
\begin{subfigure}{.5\textwidth}
  \centering
  \includegraphics[width=.9\linewidth]{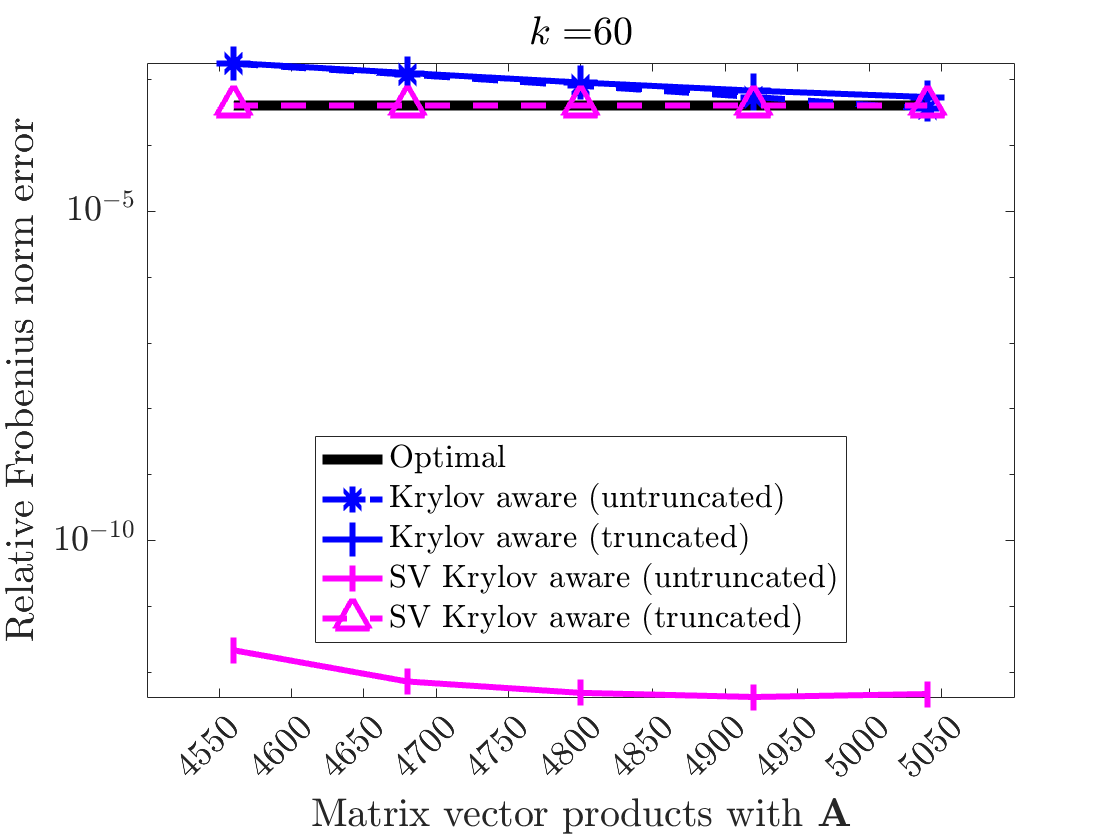}  
  \caption{Example from \Cref{section:exponential_integrator}}
\end{subfigure}
\begin{subfigure}{.5\textwidth}
  \centering
  \includegraphics[width=.9\linewidth]{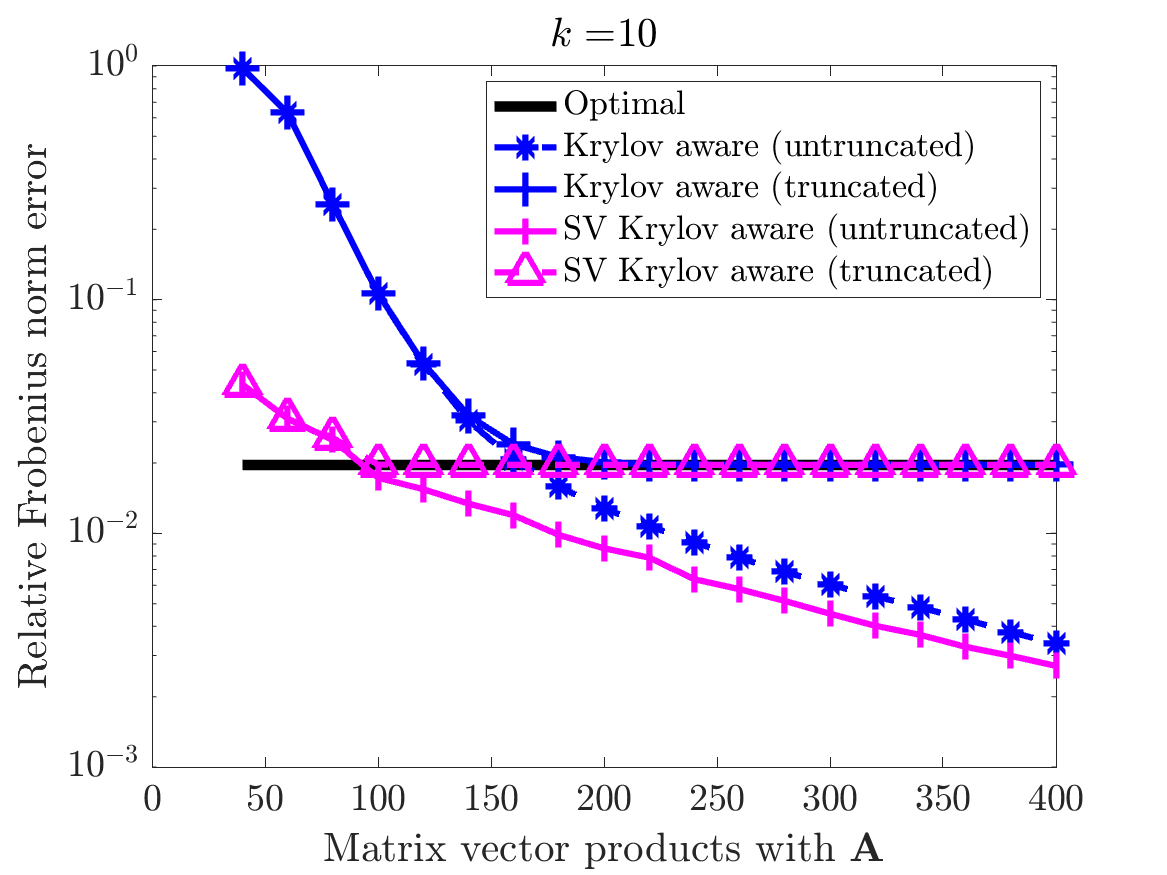} 
  \caption{Example from \Cref{section:estrada}}
\end{subfigure}
\begin{subfigure}{.5\textwidth}
  \centering
  \includegraphics[width=.9\linewidth]{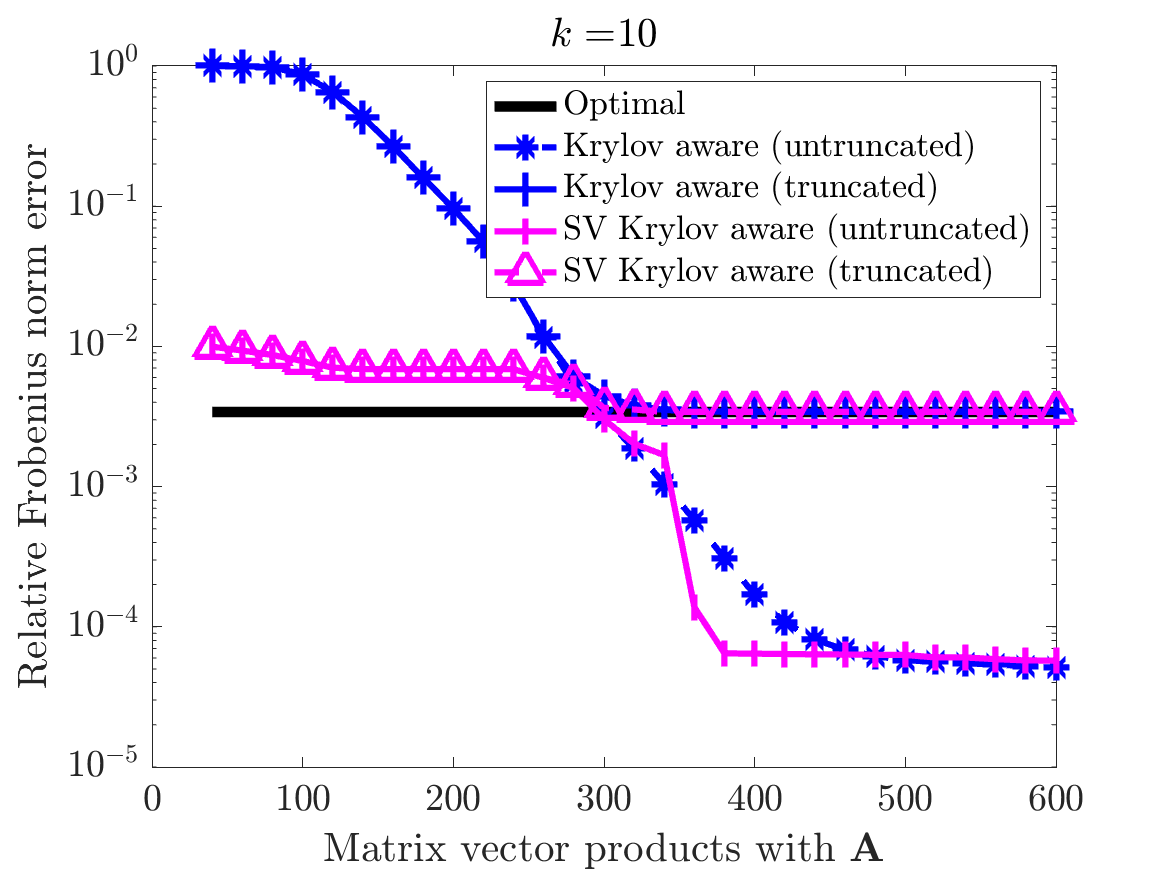}  
  \caption{Example from \Cref{section:quantum_spin}}
\end{subfigure}
\begin{subfigure}{.5\textwidth}
  \centering
  \includegraphics[width=.9\linewidth]{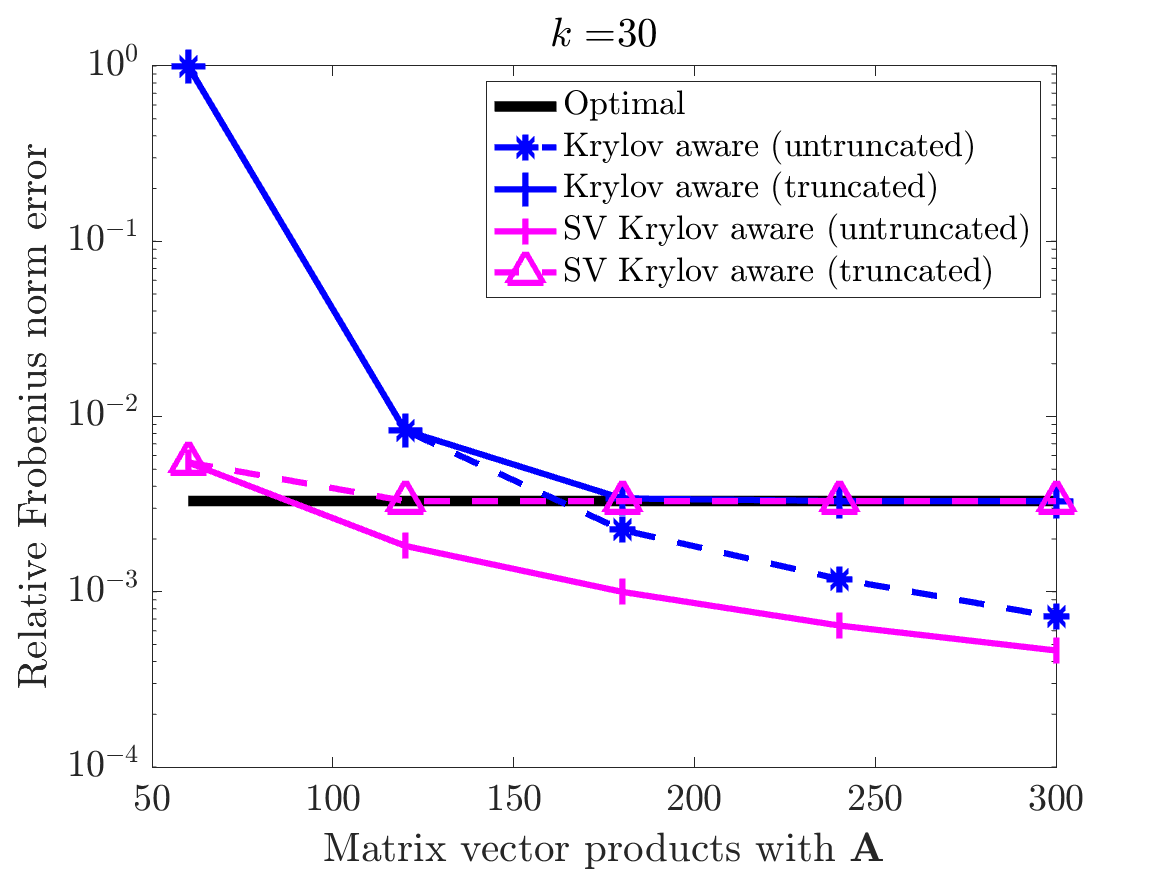}  
  \caption{Example from \Cref{section:synthetic_log}}
\end{subfigure}
\caption{\rev{Comparing relative error  \eqref{eq:relative_error} for the the approximations returned by \Cref{alg:krylow} without truncation (Krylov aware (untruncated)), \Cref{alg:krylow} with truncation back to rank $k$ (Krylov aware (truncated)), \Cref{alg:krylow_single} without truncation (SV Krylov aware (untruncated)), \Cref{alg:krylow_single} with truncation back to rank $k$ (SV Krylov aware (truncated)). The black line shows the optimal rank $k$ approximation relative Frobenius norm error. The rank parameter $k$ is visible as titles in the figures.}}
\label{fig:relative_errors_svk}
\end{figure}

\section{Conclusion}\label{section:conclusion}
In this work we have provided an analysis of the Krylov-aware low-rank approximation presented in \cite{chen_hallman_23}. In particular, we have shown that the Krylov-aware low-rank approximation to $f(\bm{A})$ is never worse than the approximation returned by the randomized SVD on $f(\bm{A})$, up to a polynomial approximation error factor\rev{, and we have presented an error analysis of the method}. Numerical experiments further demonstrated that the Krylov-aware approach is often substantially more accurate than randomized SVD for approximating $f(\bm{A})$.

\rev{However, as highlighted in \cite[Chapter 10]{thesis}, when using the same number of matvecs with $\bm{A}$ the approximation \begin{equation}\label{eq:direct}
    f(\bm{A}) \approx \bm{Q}_{s} \llbracket f(\bm{T}_{s})\rrbracket_k \bm{Q}_{s}^\T,  \quad  \bm{T}_s = \bm{Q}_s^\T \bm{A} \bm{Q}_s,
\end{equation} 
where $\bm{Q}_s$ is an orthonormal basis for a Krylov subspace, can be even more accurate than the Krylov-aware low-rank approximation, at the same computational cost. This scheme corresponds to setting $r = 0$ in \Cref{alg:krylow}. It is also closely related to the method proposed in \cite{persson_kressner_23,funnystrom2} for operator monotone functions. }

\rev{While the analysis developed in this works does not yet provide theoretical justification for \eqref{eq:direct}, extending the theory to cover \eqref{eq:direct} remains an important direction for future work. Nonetheless, our hope is that the present analysis serves as a stepping stone toward such a result. }

\begin{paragraph}{Acknowledgements}
This work was done while David Persson was a PhD-student at EPFL and supported by the SNSF research project \textit{Fast algorithms from low-rank updates}, grant number: 200020\_178806. This work is also contained in his PhD-thesis \cite{thesis}. Tyler Chen and Christopher Musco were partially supported by NSF Award \#2045590. The authors thank Daniel Kressner for helpful comments on this work.
\end{paragraph}


\bibliographystyle{siam}
\bibliography{bibliography}

\begin{appendices}
\section{Proofs of \Cref{theorem:rsvd_like_bound} and \Cref{theorem:fast_convergence}}\label{section:appendix}
\begin{proof}[Proof of \Cref{theorem:rsvd_like_bound}]
    First note that by a standard Chebyshev interpolation bound \cite[Lecture 20]{stewart}
    \begin{equation*}
        \inf\limits_{p \in \mathbb{P}_{2r+1}}\|\exp(x)-p(x)\|_{L^{\infty}([\lambda_{\min},\lambda_{\max}])}  \leq \frac{\gamma_{1,n}^{2r+r}}{2^{4r+3}(2r+2)!}\|\exp(\bm{A})\|_2.
    \end{equation*}
    Hence,
    \begin{equation*}
        4\sqrt{\ell s}  \inf\limits_{p \in \mathbb{P}_{2r+1}}\|f(x)-p(x)\|_{L^{\infty}([\lambda_{\min},\lambda_{\max}])} \leq 4\sqrt{\ell s} \frac{\gamma_{1,n}^{2r+2}}{(2r+2)!} \|\exp(\bm{A})\|_2.
    \end{equation*}
    We proceed with bounding $\mathcal{E}(s;\exp(x))$. Let $p(x) = \sum\limits_{i=1}^{s-1}\frac{(x-\lambda_{\min})^i}{i!}$. Then for $x \in [\lambda_{\min},\lambda_{\max}]$ we have $0 \leq p(x) \leq \exp(x-\lambda_{\min})$. Consequently, $\|p(\bm{\Lambda}_{n \setminus k })\|_{\F} \leq \|\exp(\bm{\Lambda}_{n \setminus k })\|_{\F} \exp(-\lambda_{\min})$. Furthermore, for $x \in [\lambda_{\min},\lambda_{\max}]$ we have
    \begin{align*}
       0 \leq  \exp(x-\lambda_{\min}) - p(x) 
       \leq \frac{(x-\lambda_{\min})^s}{s!} \exp(x-\lambda_{\min}).
    \end{align*}
    Note that by the assumption on $s$ and a Stirling approximation $s! \geq \sqrt{2\pi s} \left(\frac{s}{e}\right)^s$ \cite{stirling} we have $\frac{\gamma_{1,n}^s}{s!} < 1$.
    Hence,
    \begin{align}\label{eq:bound}
        \max\limits_{i=1,\ldots,k} \left|\frac{\exp(\lambda_i)}{p(\lambda_i)}\right| &= \exp(\lambda_{\min})\max\limits_{i=1,\ldots,k} \left|\frac{\exp(\lambda_i - \lambda_{\min})}{p(\lambda_i)}\right| 
        \nonumber\\&
        = \exp(\lambda_{\min}) \left(1- \frac{(\lambda_{\max}-\lambda_{\min})^s}{s!}\right)^{-1}. 
    \end{align}
    Therefore, $\mathcal{E}(s;\exp(x))$ is bounded above by 
    \begin{equation}\label{eq:bound2}
        \mathcal{E}(s;\exp(x)) \leq \frac{1}{(1- \frac{(\lambda_{\max}-\lambda_{\min})^s}{s!})^2} \|\exp(\bm{\Lambda}_{n \setminus k})\|_{\F}^2 = \frac{1}{(1- \frac{\gamma_{1,n}^s}{s!})^2} \|\exp(\bm{\Lambda}_{n \setminus k})\|_{\F}^2.
    \end{equation}
    Plugging the inequalities  \eqref{eq:bound} and \eqref{eq:bound2} into \Cref{theorem:krylov_aware} yields the desired inequality. 
\end{proof}

\begin{proof}[Proof of \Cref{theorem:fast_convergence}]
    Bounding $4\sqrt{\ell s}\inf\limits_{p \in \mathbb{P}_{2r+1}}\|\exp(x)-p(x)\|_{L^{\infty}([\lambda_{\min},\lambda_{\max}])}$ is done identically to in \Cref{theorem:rsvd_like_bound}. We proceed with bounding $\mathcal{E}(s;\exp(x))$. Define $\rev{\eta}(x) = \frac{x-\lambda_{k+1}}{\lambda_{k+1} - \lambda_n}$. Define the polynomial $p(x) = (1+x-\lambda_{n})T_{s-2}\left(1 + 2 \rev{\eta}(x)\right)$, where $T_{s-2}$ is the Chebyshev polynomial of degree $s-2$.\footnote{\rev{Polynomials similar to $p(x)$ are commonly used in the analysis of eigenvalue approximation using Krylov subspace methods; see \cite{shmuel, paige1971computation, saad1980rates, underwood1975iterative}.}}
    Hence, recalling the definition \cref{eqn:min_ratio} of $\mathcal{E}(s;\exp(x))$, since $0\leq 1+x \leq \exp(x)$ for $x\geq 0$ and $|T_{s-2}(x)| \leq 1$ for $x \in [-1,1]$ we have,
    \[
    \|p(\bm{\Lambda}_{n \setminus k})\|_{\F}
    \leq \|\exp(\bm{\Lambda}_{n \setminus k} - \lambda_{n}\bm{I})\|_{\F}\]
    Hence, using that $\frac{\exp(x)}{1+x} \leq \exp(x)$ for $x \geq 0$ we get
    \begin{align}
        \mathcal{E}(s;\exp(x)) 
        &\leq \|p(\bm{\Lambda}_{n \setminus k})\|_{\F}^2 \max\limits_{i=1,\ldots,k} \left|\frac{\exp(\lambda_i)}{p(\lambda_i)}\right|^2 
        \nonumber\\&\leq 
        \|\exp(\bm{\Lambda}_{n \setminus k}-\lambda_{n}\bm{I})\|_{\F}^2\max\limits_{i=1,\ldots,k} e^{2\lambda_{n}} \left|\frac{\exp(\lambda_i-\lambda_{n})}{p(\lambda_i)}\right|^2 
        \nonumber\\&=\|\exp(\bm{\Lambda}_{n \setminus k})\|_{\F}^2\max\limits_{i=1,\ldots,k} \left|\frac{\exp(\lambda_i-\lambda_{n})}{p(\lambda_i)}\right|^2
        \nonumber\\&\leq\|\exp(\bm{\Lambda}_{n \setminus k})\|_{\F}^2\max\limits_{i=1,\ldots,k} \left|\frac{\exp(\lambda_i-\lambda_{n})}{T_{s-2}\left(1 + 2 \rev{\eta}(\lambda_i)\right)}\right|^2.
        \label{eqn:Esexp_bd}
    \end{align}
    First note that for $|x| \geq 1$ we have
    \begin{equation*}
        T_{s-2}(x) = \frac{1}{2}\left(\left(x + \sqrt{x^2-1}\right)^{s-2} + \left(x - \sqrt{x^2-1}\right)^{s-2} \right).
    \end{equation*}
    Hence, for $\rev{\eta} \geq 0$ we have $|T_{s-2}(1+2\rev{\eta})| \geq \frac{1}{2}(1+2\rev{\eta} + 2\sqrt{\rev{\eta} + \rev{\eta}^2})^{s-2} =: \frac{1}{2}h(\rev{\eta})^{s-2}$. Therefore,
    \begin{align*}
        \max\limits_{i=1,\ldots,k} \left|\frac{e^{\lambda_i-\lambda_{n}}}{T_{s-2}\left(1 + 2 \rev{\eta}(\lambda_i)\right)}\right|^2 &\leq 4 \max\limits_{i=1,\ldots,k} \left|\frac{\exp(\lambda_i-\lambda_{n})}{h(\rev{\eta}(\lambda_i))^{s-2}}\right|^2\\
        &= 4 \max\limits_{i=1,\ldots,k} \left(\exp(\lambda_i - \lambda_n - (s-2)\log(h(\rev{\eta}(\lambda_i))))\right)^2.
    \end{align*}
    Note that the function 
    \begin{equation*}
        g(x) := x - \lambda_n - (s-2)\log(h(\rev{\eta}(x)))
    \end{equation*}
    is convex in $[\lambda_k,\lambda_1]$ because $\log(h(x))$ is concave for $x \geq 0$ and $\rev{\eta}(x)$ is linear positive function for any $x \geq \lambda_{k}$. Hence, the maximum of $g$ on the interval $[\lambda_k,\lambda_1]$ is attained at either $x = \lambda_k$ or $x = \lambda_1$. The maximum of $g$ is attained at $x = \lambda_k$ whenever $g(\lambda_1) \leq g(\lambda_k)$ which happens whenever
    \begin{equation}\label{eq:scondition}
        \frac{\lambda_1 - \lambda_k}{\log\left(\frac{h(\rev{\eta}(\lambda_1))}{h(\rev{\eta}(\lambda_k))}\right)} + 2 \leq s.
    \end{equation}
    Note that we have whenever $x \geq y \geq 0$ we have $\frac{h(x)}{h(y)} \geq \frac{1 + x}{1+y}$, which can be seen by noting that $\frac{h(x)}{1+x}$ is an increasing function. Hence, \eqref{eq:scondition} holds whenever
    \begin{equation*}
        \frac{\lambda_1 - \lambda_k}{\log\left(\frac{\lambda_1 - \lambda_{n}}{\lambda_k - \lambda_{n}}\right)} + 2  \leq s,
    \end{equation*}
    which is our assumption on $s$. Hence, 
    \begin{equation}\label{eq:upperbound}
         \max\limits_{i=1,\ldots,k} \left|\frac{\exp(\lambda_i-\lambda_{n})}{T_{s-2}\left(1 + 2 \rev{\eta}(\lambda_i)\right)}\right|^2 \leq 4 \frac{\exp(2(\lambda_k - \lambda_n))}{h(\rev{\eta}(\lambda_k))^{2(s-2)}}.
    \end{equation}
    Now we use the argument from \cite[p.21]{MM15}, which shows that $h(\rev{\eta})^{(s-2)} \geq 2^{\sqrt{\min\{1,2\rev{\eta}\}}(s-2) - 1}$. Plugging this inequality into \cref{eq:upperbound} and then \cref{eqn:Esexp_bd} and  using the definition for $\gamma_{i,j}$ yields
    \begin{equation*}
        \mathcal{E}(s;\exp(x)) \leq 16 \exp(2\gamma_{k,n}) 2^{-2(s-2) \sqrt{\min\left\{1,2 \frac{\gamma_{k,k+1}}{\gamma_{k+1,n}}\right\}}} \|\exp(\bm{\Lambda}_{n\setminus k})\|_\F^2.
    \end{equation*}
    \rev{Noting $\exp(2\gamma_{k,n}) \leq 2^{3\gamma_{k,n}}$ and inserting the definition of $\Gamma$ yields the desired result.}

    \rev{The second statement follows directly. Setting $\ell = 2k+1$ and $s = O\left(\frac{\log(1/\varepsilon) + \gamma_{k,n}}{\sqrt{\Gamma}}\right)$ yields $2^{-2(s-2)\sqrt{\Gamma} + 3\gamma_{k,n}}\frac{80k}{\ell - k - 1} \leq \varepsilon$. Selecting $r$ as prescribed yields makes the polynomial approximation term $\leq \tau \|\exp(\bm{A})\|_2$ by a Stirling approximation on $(2r+2)!$ \cite{stirling}, yielding the desired result.}
\end{proof}

\section{\rev{Proof of \Cref{theorem:singlevector}}}\label{section:appendixB}
\rev{To prove \Cref{lemma:intermediatebound}, we begin with the following simple lemma. }

\rev{\begin{lemma}\label{lemma:intermediatebound}
    Let $\bm{\omega}_1 \in \mathbb{R}^k$ and $\bm{\omega}_2$ be standard Gaussian random vectors, and let $\bm{\omega}_{1,i}$ denote the $i^{\text{th}}$ entry of $\bm{\omega}_1$. For any fixed matrix $\bm{B}$, the following inequality holds with probability at least $1-\delta$
    \begin{equation*}
        \frac{\|\bm{B} \bm{\omega}_2 \|_2^2}{\min\limits_{i=1,\ldots,k}\bm{\omega}_{1,i}^2} \leq  \widehat{C}_{\delta}k^2\|\bm{B}\|_\F^2, 
    \end{equation*}
    where $\widehat{C}_{\delta} = \frac{2 \pi (1+6\log(2/\delta))}{\delta^2}$ is as in \Cref{theorem:singlevector}.
\end{lemma}}

\rev{\begin{proof}
    By \cite[Lemma 3.1]{MeyerMuscoMusco:2024} we have $\frac{1}{\min\limits_{i=1,\ldots,k}\bm{\omega}_{1,i}^2} \leq  \frac{2 \pi k^2}{\delta^2}$ with probability $1-\delta/2$. Furthermore, by a one-sided version of \cite[Theorem 1]{cortinoviskressner} we have $\mathbb{P}(\|\bm{B}\bm{\omega}_2\|_2^2 \geq  (1+x)\|\bm{B}\|_\F^2) \leq \exp\left(-\frac{x^2}{4(1+x)}\right)$. Setting $x = 6\log(2/\delta)$ gives $\mathbb{P}(\|\bm{B}\bm{\omega}_2\|_2^2 \geq  (1+x)\|\bm{B}\|_\F^2) \leq \exp\left(-\frac{x^2}{4(1+x)}\right) \leq \delta/2$. Applying the union bound yields the desired result.
\end{proof}}

\rev{With \Cref{lemma:intermediatebound} at hand, we are ready to prove \Cref{theorem:singlevector}.}

\rev{\begin{proof}[Proof of \Cref{theorem:singlevector}]
Define
    \begin{equation*}
        \widehat{\bm{\Omega}} = \begin{bmatrix} \bm{\omega} & \bm{A} \bm{\omega} & \cdots & \bm{A}^{k-1} \bm{\omega} \end{bmatrix}, \quad \widehat{\bm{\Omega}}_k = \bm{U}_{k}^\T \widehat{\bm{\Omega}}, \quad \widehat{\bm{\Omega}}_{n\setminus k} = \bm{U}_{n \setminus k}^\T \widehat{\bm{\Omega}}
    \end{equation*}
    By the discussion in \Cref{section:singlevector} it suffices to show that for any polynomial $p \in \mathbb{P}_{s-1}$ we have
    \begin{equation*}
        \|p(\bm{\Lambda}_{n \setminus k}) \widehat{\bm{\Omega}}_{n \setminus k} \widehat{\bm{\Omega}}_k^{-1}\|_\F^2 \leq \widehat{C}_{\delta}\frac{k^4}{\gamma_{\min}^{2(k-1)}} \|p(\bm{\Lambda}_{n \setminus k})\|_\F^2. 
    \end{equation*}
    Our proof follows the same idea in \cite[Theorem 3.1]{MeyerMuscoMusco:2024}, but we take care to not introduce any dependencies on the matrix size $n$. We have
    \begin{align*}
        \|p(\bm{\Lambda}_{n \setminus k}) \widehat{\bm{\Omega}}_{n \setminus k} \widehat{\bm{\Omega}}_k^{-1}\|_\F^2 &\leq k \|p(\bm{\Lambda}_{n \setminus k}) \widehat{\bm{\Omega}}_{n \setminus k} \widehat{\bm{\Omega}}_k^{-1}\|_2^2\\
        & = k \max\limits_{\bm{x} \neq \bm{0}} \frac{\|p(\bm{\Lambda}_{n \setminus k}) \widehat{\bm{\Omega}}_{n \setminus k} \widehat{\bm{\Omega}}_k^{-1} \bm{x}\|_2^2}{\|\bm{x}\|_2^2} \\
        &= k \max\limits_{\bm{y} \neq \bm{0}} \frac{\|p(\bm{\Lambda}_{n \setminus k}) \widehat{\bm{\Omega}}_{n \setminus k} \bm{y}\|_2^2}{\|\widehat{\bm{\Omega}}_k\bm{y}\|_2^2}\\
        &= k \max\limits_{g \in \mathbb{P}_{k-1}, g \neq 0} \frac{\|p(\bm{\Lambda}_{n \setminus k}) g(\bm{\Lambda}_{n \setminus k})\bm{\omega}_{n \setminus k} \|_2^2}{\|g(\bm{\Lambda}_k) \bm{\omega}_k\|_2^2},
    \end{align*}
    where $\bm{\omega}_k = \bm{U}_k^\T \bm{\omega}$ and $\bm{\omega}_{n \setminus k} = \bm{U}_{n \setminus k}^\T \bm{\omega}$ are (independent) standard Gaussian random vectors. Using $p(\bm{\Lambda}_{n \setminus k}) g(\bm{\Lambda}_{n \setminus k}) = g(\bm{\Lambda}_{n \setminus k}) p(\bm{\Lambda}_{n \setminus k})$, submultiplicativity of the Euclidean norm, and \Cref{lemma:intermediatebound} we have with probability $\geq 1-\delta$
    \begin{align*}
        \|p(\bm{\Lambda}_{n \setminus k}) \widehat{\bm{\Omega}}_{n \setminus k} \widehat{\bm{\Omega}}_k^{-1}\|_\F^2 &\leq k \max\limits_{g \in \mathbb{P}_{k-1}, g \neq 0} \frac{\|p(\bm{\Lambda}_{n \setminus k}) g(\bm{\Lambda}_{n \setminus k})\bm{\omega}_{n \setminus k} \|_2^2}{\|g(\bm{\Lambda}_k) \bm{\omega}_k\|_2^2}\\
        & \leq k \frac{\|p(\bm{\Lambda}_{n \setminus k})\bm{\omega}_{n \setminus k}\|_2^2}{\min\limits_{i=1,\ldots,k} \bm{\omega}_{k,i}^2} \max\limits_{g \in \mathbb{P}_{k-1}, g \neq 0} \frac{\|g(\bm{\Lambda}_{n \setminus k})\|_2^2}{\|g(\bm{\Lambda}_{k})\|_\F^2}\\
        & \leq \widehat{C}_{\delta}k^3\max\limits_{g \in \mathbb{P}_{k-1}, g\neq 0} \frac{\|g(\bm{\Lambda}_{n \setminus k})\|_2^2}{\|g(\bm{\Lambda}_{k})\|_\F^2}.\numberthis \label{eq:remainderbound}
    \end{align*}
    As done in the proof of \cite[Theorem 3.1]{MeyerMuscoMusco:2024}, we write $g$ as an interpolating polynomial
    \begin{equation*}
        g(x) = \sum\limits_{i=1}^{k} g(\lambda_i) \ell_i(x), \quad \ell_i(x) = \prod\limits_{t = 1, t \neq i}^k \frac{x-\lambda_t}{\lambda_i - \lambda_t}.
    \end{equation*}
    Note that a simple bound yields $|\ell_{i}(x)| \leq \frac{1}{\gamma_{\min}^{k-1}}$. Hence, 
    \begin{align*}
        \max\limits_{g \in \mathbb{P}_{k-1}, g\neq 0} \frac{\|g(\bm{\Lambda}_{n \setminus k})\|_2^2}{\|g(\bm{\Lambda}_{k})\|_\F^2} &= \max\limits_{g \in \mathbb{P}_{k-1}, g\neq 0} \frac{\max\limits_{j = k+1,\ldots,n} \left|\sum\limits_{i=1}^{k} g(\lambda_i) \ell_i(\lambda_j)\right|^2}{\sum\limits_{i=1}^k g(\lambda_i)^2}\\
        & \leq \max\limits_{j=k+1,\ldots,n} \left(\sum\limits_{i=1}^k |\ell_i(\lambda_j)|\right)^2 \\
        &\leq \max\limits_{x \in [\lambda_{\min},\lambda_{\max}]} \sum\limits_{i=1}^k |\ell_i(x)|^2\\
        & \leq \frac{k}{\gamma_{\min}^{2(k-1)}}.\numberthis \label{eq:gammabound}
    \end{align*}
    where we used the Cauchy-Schwarz inequality and that the bound for $|\ell_i(x)|$. Combining \eqref{eq:gammabound} with \eqref{eq:remainderbound} yields the desired result.
    \end{proof}}


\end{appendices}

\end{document}